\documentclass[bj, preprint]{imsart}

\RequirePackage[OT1]{fontenc}
\RequirePackage{amsthm,amsmath}
\RequirePackage[numbers]{natbib}
\RequirePackage[colorlinks,citecolor=blue,urlcolor=blue]{hyperref}

\usepackage[dvips]{graphicx}
\usepackage{verbatim}
\usepackage{bbm}
\usepackage{graphicx} 
\usepackage[cmex10]{amsmath}
\usepackage{amssymb}
\usepackage{mathrsfs}
\usepackage{verbatim}
\usepackage{epstopdf}

\pubyear{2013}
\volume{0}
\issue{0}
\firstpage{1}
\lastpage{29}
\arxiv{1210.2029}

\startlocaldefs
\numberwithin{equation}{section}
\theoremstyle{plain}

\endlocaldefs

\newcommand{\ind}[1]{\mathbbm{1}_{\{#1\}}}   

\newcommand{\uei}{\underline{\epsilon}^k}
\newcommand{\oei}{\bar{\epsilon}^k}

\newcommand{\ei}{\epsilon^k}

\newcommand{\uli}{\underline{\Lambda}^k}
\newcommand{\oli}{\bar{\Lambda}^k}
\newcommand{\li}{\Lambda^k}

\newcommand{\oDi}{\bar{\Delta}^k}
\newcommand{\uDi}{\underline{\Delta}^k}
\newcommand{\Di}{\Delta^k}
\newcommand{\D}{\Delta}

\newcommand{\uRi}{\underline{R}^k}
\newcommand{\oRi}{\overline{R}^k}

\newcommand{\cF}{\mathscr{F}}
\newcommand{\cT}{\mathcal{T}}
\newcommand{\cN}{\mathcal{N}}
\newcommand{\cM}{\mathcal{M}}
\newcommand{\cS}{\mathcal{S}}
\newcommand{\cJ}{\mathcal{J}}
\newcommand{\cR}{\mathcal{R}}

\newcommand{\cC}{\mathcal{C}}

\newcommand{\ccL}{\mathcal{L}}

\newcommand{\Exp}{\mathsf{E}}
\newcommand{\Pro}{\mathsf{P}}

\newcommand{\cFt}{ {\cF}_{t}}
\newcommand{\cus}{ \cS} 
\newcommand{\dcus}{\tilde{\cus}}

\newcommand{\calo}{\mathcal{O}}

\newtheorem{theorem}{Theorem}

\newtheorem{lemma}{Lemma}

\newcommand{\ignore}[1]{{}}

\begin{document}

\begin{frontmatter}
\title{Bandwidth and Energy Efficient Decentralized Sequential Change Detection\thanksref{T1}}
\runtitle{Decentralized Sequential Change Detection}
\thankstext{T1}{This work was supported in part by the US National Science Foundation
under Grant CIF1064575.}

\begin{aug}
\author{\fnms{Georgios} \snm{Fellouris}\thanksref{a,e1}\ead[label=e1,mark]{fellouri@illinois.edu}}
\and
\author{\fnms{George V.} \snm{Moustakides}\thanksref{b,e2}\ead[label=e2,mark]{moustaki@upatras.gr}}

\address[a]{Department of  Statistics, University  of Illinois, Urbana-Champaign, IL 61820, USA.
\printead{e1}}

\address[b]{Department of Electrical and  Computer Engineering, University of Patras, 26500 Rion, Greece.\\
\printead{e2}}

\runauthor{Fellouris and Moustakides}
\affiliation{University of Illinois  and  University of Patras}

\end{aug}

\begin{abstract}
The problem of decentralized sequential change detection is considered, where an abrupt change occurs in an area monitored by a number of sensors;
the sensors transmit their data to  a fusion center, subject to bandwidth and energy constraints, and  the fusion center is responsible 
for detecting the change as soon as possible. A novel sequential detection rule is proposed that requires communication from the sensors at random times and transmission of only low-bit messages, on which the fusion center runs in parallel a CUSUM test. The second-order asymptotic optimality of the proposed scheme is established  both in discrete and in continuous time. Specifically, it is shown that the inflicted performance loss (with respect to the optimal detection rule that uses the complete sensor observations) is asymptotically bounded as the rate of false alarms goes to 0, for any fixed rate of communication.
When the rate of communication from the sensors is asymptotically low, the proposed scheme remains first-order asymptotically optimal. Finally, simulation experiments illustrate its efficiency and its superiority over a decentralized detection rule that relies on communication at deterministic times. 
\end{abstract}



\end{frontmatter}

\section{Introduction}
Suppose that an area is being monitored by a number of sensors which transmit their observations to a central location, that we will call fusion center. 
At some unknown time, an abrupt disorder occurs, such as an unexpected intrusion, and changes the dynamics of the observed processes in all sensors simultaneously. The goal is to raise an alarm at the fusion center  as soon as possible after the occurrence of the change.
When the sensors transmit their complete observations to the fusion center, this is the classical problem of sequential change detection, for exhaustive reviews on which we refer to  \cite{niki}, \cite{poor},  \cite{lairoyal}, \cite{shir}, \cite{polu}. However, classical detection rules typically are not applicable in modern application areas, such as mobile and wireless communications and distributed surveillance systems. In such systems, the sensors are typically low-power devices whose links with the fusion center are characterized by limited communication bandwidth \cite{vij},\cite{veera}. Thus, in order to preserve the robustness of the network, it is necessary to limit the overall communication load and, in particular, the transmission activity of each sensor. This primarily implies a \textit{quantization} constraint, i.e., each sensor should  transmit a small number of bits each time it communicates with the fusion center, but also a \textit{rate} constraint, i.e., each sensor should communicate with the fusion center at a lower rate than its sampling rate. As a result, 
before constructing a sequential detection rule at the fusion center, the designer must first decide what kind of information  should be transmitted from the sensors, taking into account the above communication constraints. In what follows, we will call detection rules  that respect such constraints \textit{decentralized}, in contrast to the \textit{centralized} ones that require knowledge  of the 
full sensor observations. 

Most papers in the decentralized literature (see, e.g., \cite{crow}, \cite{veer}, \cite{veera}, \cite{tart}) assume that each sensor transmits a quantized version of \textit{every} observation it takes, i.e., the communication rate is equal to the sampling rate.  For a discussion on one-shot schemes, where each sensor transmits to the fusion center a single bit \textit{at most once}, we refer to \cite{moustdec}. A decentralized detection rule which enjoys an asymptotic optimality property was proposed by Mei \cite{mei}, however the performance of this scheme in practice is often worse than that of asymptotically suboptimal detection rules. Thus, it has been an open problem to find an asymptotically optimal decentralized detection rule that is also efficient in practice. 

The main contribution of this work is that we propose such a rule.  Specifically, we suggest that each sensor communicates with the fusion center at  stopping times of its local filtration; at every communication, it transmits a  \textit{low-bit} message which ``summarizes'' the evolution of its local sufficient statistic since the previous communication; the fusion center, in parallel, runs a CUSUM test on the transmitted messages in order to detect the change. 
For similar communication schemes in the context of decentralized sequential hypothesis testing we refer to \cite{fel} and \cite{yil}. The design and analysis of the proposed scheme, that we call D-CUSUM, is different in discrete and continuous time. However, in both cases we establish a \textit{second-order} asymptotic optimality property, that is stronger than the first-order asymptotic optimality of the detection rule in \cite{mei}. In particular, we show that the performance loss of D-CUSUM with respect to the optimal centralized CUSUM remains bounded as the period of false alarms goes to infinity. Moreover, we show that D-CUSUM remains first-order  asymptotically optimal even when it induces an asymptotically low communication rate and there is an asymptotically  large number of sensors. Simulation experiments suggest that these strong theoretical properties are also accompanied by  very good performance in practice and that D-CUSUM is much more efficient than a similar, CUSUM-based decentralized detection rule that relies on communication from the sensors  at deterministic times. 

In what follows, in Section 2, we formulate the problem of (decentralized) sequential change detection and  describe the main decentralized schemes in the literature. In Section 3, we define and analyze the proposed scheme both in continuous and in discrete time. In Section 4, we summarize and discuss an extension in the case of correlated sensors. The proof of all results,  as well as some supporting lemmas, are presented in Appendices A-E.

\section{Sequential Change Detection}
Let $\{(\xi_{t}:=\xi_{t}^{1},\ldots, \xi_{t}^{K})\}$  be a $K$-dimensional stochastic process, where $\xi_{0}^{k}:=0$ and $\xi^{k}$ is the observed process at sensor $k$, $1 \leq k \leq K$. We denote by $\{\cFt^{k}\}$  the local filtration at sensor $k$ and by  $\{\cFt\}$  the global filtration, i.e., $\cFt^{k}:= \sigma(\xi_{s}^k, \, 0  \leq s \leq t)$  and  $\cF_{t}:= \vee_{k} \cF_{t}^{k}$. Time may be either discrete $(t \in \mathbb{N})$ or continuous $(t \in [0, \infty))$ and in the latter case all filtrations are considered to be right-continuous. We assume that at  some unknown, deterministic time $\tau \geq 0$, the distribution of $\xi$, which we denote by $\Pro_\tau$, changes from $\Pro_{\infty}$ to $\Pro_{0}$, where $\Pro_{0}$ and $\Pro_{\infty}$ are two completely specified, locally equivalent  probability measures on the canonical space of  $\xi$. In other words,  $\Pro_\tau$ coincides with $\Pro_{\infty}$ when both measures are restricted to $\cFt$ and $t \leq \tau$, whereas for  $t> \tau$ we can define the following log-likelihood ratio process
$$
u_{t}-u_{\tau}:=\log \frac{\text{d} \Pro_{\tau}}{\text{d} \Pro_{\infty}} \Big|_{\cFt}  , \quad t \geq \tau; \quad u_{0}:=0.                 
$$

\subsection{The centralized setup}
In the centralized setup,  where the fusion center has access to all sensor observations, the problem is to find an $\{\cFt\}$-stopping time $\cT$ 
that has  small detection delay and rare false alarms, i.e.,  $\cT$ should take large values under $\Pro_{\infty}$ and $\cT-\tau$ small values under $\Pro_{\tau}$. There are different approaches in how to quantify  detection delay and false alarms, such as the Bayesian formulation due to Shiryaev \cite{shiry} (see also \cite{bei2}, \cite{pes}, \cite{gap}, \cite{savas}, \cite{sezer}) or the minimax formulation due to Pollak \cite{poll} (see also \cite{polu}, \cite{tpp}). In this work, we focus on the formulation suggested by Lorden \cite{lord}, where the performance of a detection rule $\cT$ is measured by its worst-case (with respect to $\tau$)  conditional expected delay given the worst possible history of observations up to $\tau$,
\begin{equation} \label{lord}
\cJ_{\text{L}}[\cT] =  \sup_{\tau \geq 0} \, \text{ess\,sup} \; \Exp_{\tau} [  (\cT-\tau)^{+}   | \cF_{\tau} ],
\end{equation}
and an optimal detection rule is a solution to the following optimization problem
\begin{equation} \label{lord_crit}
\inf_{\cT} \cJ_{\text{L}}[\cT] ~ \text{when} ~  \Exp_{\infty} [\cT] \geq \gamma ,
\end{equation} 
where $\gamma>0$. In other words,  the goal in this approach is to minimize the detection delay under the worst-case scenario 
with respect to both the changepoint and the history of observations before the change, 
while controlling the period of false alarms above  a desired level, $\gamma$. It is well known  (see \cite{moust}, \cite{moustext}) 
that when $\{u_{t}\}_{t \in \mathbb{N}}$ is a random walk, the solution to this problem  is given by Page's \cite{page} Cumulative Sums (CUSUM) test,
\begin{equation}  \label{cusum}
\cus:= \inf \{ t \geq 0: y_{t}  \geq \nu \} , ~ \text{where} \quad  y_{t}:= u_{t} - \inf_{0 \leq s < t} u_{s},
\end{equation}
and $\nu$ is defined so that the false alarm constraint in (\ref{lord_crit}) be satisfied with equality, i.e., $\Exp_{\infty} [\cus]=\gamma$. This exact (i.e., non-asymptotic) optimality of the CUSUM test can be extended to a much richer class of dynamics if we adopt an idea of Liptser and Shiryaev \cite{lip} and measure detection delay and period of false alarms not in terms of actual time, but in terms of  Kullback-Leibler divergence. Indeed, working similarly to \cite{moustito}, we replace the performance measure $\cJ_{L}$ by 
\begin{equation} \label{mod}
\cJ[\cT] :=  \sup_{\tau \geq 0} \, \text{ess\,sup} \; \Exp_{\tau} [ ( u_{\cT}-  u _{\tau} )\ind{\cT>\tau} | \cF_{\tau} ] 
\end{equation}
and define an optimal detection rule as a solution to 
 \begin{equation} \label{mod_crit}
\inf_{\cT} \cJ[\cT] ~ \text{when} ~  \Exp_{\infty} [-u_{\cT}] \geq \gamma,
\end{equation} 
a problem that is equivalent to \eqref{lord_crit} when $\{u_{t}\}$ is a random walk. However,  it has been shown in \cite{moustito}, \cite{chro}  that
the CUSUM test, with threshold $\nu$ chosen so that $\Exp_{\infty} [-u_{\cus}] = \gamma$, also solves problem \eqref{mod_crit}  whenever $\{u_{t}\}$ has continuous paths and  
\begin{equation} \label{full} 
\lim_{t\rightarrow \infty} \langle u \rangle_{t}=\infty \quad \Pro_{0}, \Pro_{\infty}-\text{a.s.},
\end{equation}
where $\langle u \rangle_{t}$ is the quadratic variation of $u_t$. The latter optimality result implies that CUSUM solves Lorden's original problem (\ref{lord_crit}) whenever  $\{u_{t}\}$ has continuous paths and $\langle u\rangle_{t}$ is proportional to $t$. This is the case, for example,  when each $\xi^{k}$ is a fractional Brownian motion (fBm) with Hurst index $H$ before the change and adopts a polynomial drift term with exponent $H+1/2$ after the  change \cite{chro}. In the special case that $H=1/2$, this implies the well-known optimality of CUSUM for detecting a constant drift in a Brownian motion, established
 by Shiryaev \cite{shircus} and Beibel \cite{bei}.

\subsection{The decentralized setup}
Centralized (classical) detection rules as the CUSUM test cannot be applied in a decentralized setup, where  communication constraints must be taken into account.
In this context, before defining a detection rule at the fusion center, we must first specify a \textit{communication scheme}, that will determine the information that will be transmitted from the sensors to the fusion center. Therefore, we define a \textit{decentralized} sequential detection rule as a pair $(\{\tilde{\cFt}\}, \cT)$, where  $\cT$ is an $\{\tilde{\cFt}\}$-stopping time and 
$\{\tilde{\cFt}\}$ is a filtration of the form 
\begin{equation} \label{flow}
\tilde{\cFt}:= \sigma (( \tau^k_{n},z_{n}^k): \tau_{n}^k \leq t, k=1, \ldots, K),
\end{equation}
where each $\{\tau_{n}^k\}_{n \in \mathbb{N}}$  is the sequence of communication times for sensor $k$ and $z_{n}^k$ is the message transmitted to the fusion center at time $\tau_{n}^{k}$. Each $\tau_{n}^k$ must be an $\{\cFt^k\}$-stopping time and each  $z_{n}^k$ an  $\cF^k_{\tau_{n}^k}$-measurable random variable that takes values in a \textit{finite} set, so that a small number of bits is required for its transmission to the fusion center.  Moreover,
since many applications are characterized by limited storage capacity, we require additionally that each $z_{n}^{k}$ is measurable with respect to 
$\sigma(\xi_{s}^{k}, \;  \tau_{n-1}^{k} \leq s \leq \tau_{n}^{k})$, the $\sigma$-algebra generated by the observations at sensor $k$ between its $n-1$ and $nth$ transmission. Note that this framework forbids communication between sensors or feedback from the fusion center to the sensors. Such possibilities impose a much heavier communication load on the network and raise questions regarding the design of the network architecture, which we do not consider here. For decentralized detection rules that require feedback we refer to  \cite{veer}.

Ideally, we would like to find the best possible decentralized detection rule, performing a joint optimization over the communication scheme at the sensors and the detection rule at the fusion center.  Such an optimization problem is highly intractable, even if one makes a number of simplifying assumptions \cite{veer}. For this reason, we will use the centralized CUSUM as the ultimate benchmark and compare any decentralized detection rule against it. We can only hope that such a detection rule attains the optimal centralized performance asymptotically. Thus, if $(\{\tilde{\cFt}\},\cT)$ is an arbitrary decentralized detection rule and $\cS$  the centralized CUSUM test so that $\Exp_{\infty}[-u_{\cT}]\ge\gamma= \Exp_{\infty}[-u_{\cS}]$ for any $\gamma>0$, we will say that $\cT$ is asymptotically optimal  of \textit{first} order if $\cJ[\cT]/\cJ[\cS] \rightarrow 1$ as $\gamma \rightarrow \infty$ and of \textit{second} order  if $\cJ[\cT] - \cJ[\cS] = \calo(1)$ as $\gamma\to\infty$. Clearly,  since $\cJ[S] \rightarrow \infty$ as $\gamma \rightarrow \infty$,  second order asymptotic  optimality is a stronger property, which guarantees that the inflicted performance loss remains bounded as the rate of false alarms goes to 0.  

As it is common in the literature of decentralized sequential detection, we will assume that observations from different sensors are independent. 
Thus, if $\Pro_{\tau}^{k}$ is the distribution of $\xi^{k}$, then  
$\Pro_{\tau}= \Pro_{\tau}^{1} \times \ldots \times \Pro_{\tau}^{K}$ for any $\tau \in [0, \infty]$ and, consequently, 
$$
u_{t}:= u_{t}^{1}+ \ldots + u_{t}^{K}, \quad \text{where} \quad u_{t}^k :=\log \frac{\text{d} \Pro_{0}^{k}}{\text{d} \Pro_{\infty}^{k}} \Big|_{\cFt^k},
$$ 
for any $t \geq 0$. We also assume that the local Kullback-Leibler (KL) information numbers, $I_{0}^k:= \Exp_{0}[u_{1}^k]$ and $I_{\infty}^k:= -\Exp_{\infty}[u_{1}^k]$, are positive and finite for every $1\leq k \leq K$ and, furthermore, we define the corresponding average KL-numbers
\begin{equation} \label{klaver}
\bar{I}_{0}:= \frac{1}{K}\Exp_{0}[u_{1}]= \frac{1}{K}\sum_{k=1}^{K} I_{0}^k \quad  \text{and} \quad  \bar{I}_{\infty}:= \frac{1}{K}\Exp_{\infty}[-u_{1}]=\frac{1}{K}\sum_{k=1}^{K} I_{\infty}^k.
\end{equation} 
In the remainder of this section,  we describe the main decentralized sequential detection rules in the literature, embedding them in the above framework.
We classify them into two categories; in the first, the sensors transmit systematically compressed versions of their data to the  fusion center 
and the latter combines the received messages in order to detect the change; in the second, each sensor detects individually the change and the fusion center
combines the local  sensor decisions.

\subsubsection{Q-CUSUM} 
Suppose that each sensor transmits to the fusion center quantized versions of its local log-likelihood ratio process at deterministic, equidistant times.
Specifically, if  for each sensor  the communication period is $r$  and the available alphabet  $\{1,\ldots, b\}$, where $b\geq 2$ is an integer,
then 
\begin{equation} \label{zdd}
\tau_{n}^k=rn \; \text{and} \; 
z_n^k= \sum_{j=1}^{b} j \, \ind{\Gamma_{j-1}^k \leq  u^k_{rn}- u^k_{r(n-1)} < {\Gamma}_{j}^k},
\end{equation} 
where  $-\infty=:\Gamma_{0}^k< {\Gamma}_{1}^k<\ldots  < {\Gamma}_{b}^k:=\infty$ are fixed thresholds. This  communication scheme induces synchronous communication to the fusion center, which receives at each time $\tau_{n}^k=rn$ the $K$-dimensional vector $(z_{n}^{1}, \ldots, z_{n}^{K})$. If we additionally assume that each $\{u_{t}^k\}$ has stationary and independent increments,  then a  natural detection rule at the fusion center is the corresponding CUSUM stopping time
\begin{equation} \label{eq:cus2}
\hat{\cS} := r \cdot \inf\{n \in \mathbb{N} : \hat{y}_{n} \geq \hat{\nu}\},
\end{equation}
where the threshold $\hat{\nu}$ is chosen so that the false alarm constraint be satisfied with equality  and the  CUSUM statistic $\{\hat{y}_{n}\}$ admits the following recursion:
\begin{equation}
\hat{y}_{n}:= (\hat{y}_{n-1})^{+} + \sum_{k=1}^{K} \sum_{j=1}^{b} \left[  \ind{z_{n}^k=j} \log \frac{\Pro_{0}(z_{n}^k=j)}{\Pro_{\infty}(z_{n}^k=j)} \right] , ~ \hat{y}_{0}:=0,
\label{eq:cus1}
\end{equation}
Note that we have to multiply by $r$ in \eqref{eq:cus2} in order to return to physical time units, since the samples are acquired with a rate $1/r$.
We call this detection scheme Q-CUSUM, where Q stands for the ``quantization'' employed by this method.  This detection rule has been studied in \cite{crow}, \cite{mei}, \cite{tart}  in the case that the sensors take i.i.d. observations and each sensor communicates with the fusion center at \textit{every} observation time ($r=1$). It is easy to see that  as $\gamma \rightarrow \infty$
$$ 
\frac{\cJ[\hat{\cS}]}{\cJ[\cS]} \rightarrow \frac{r \bar{I}_{0}}{\hat{I}_{0}}, \quad \text{where} \; 
\hat{I}_{0}:= \frac{1}{K} \sum_{i=1}^{K} \sum_{j=1}^{b}  \Pro_{0}(z_{n}^k=j)\log \frac{ \Pro_{0}(z_{n}^k=j)}{\Pro_{\infty}(z_{n}^k=j)} ,
$$ 
and $\bar{I}_{0}$ is the average KL-number defined  in (\ref{klaver}), which implies that the asymptotic  performance of $\hat{\cS}$ is optimized by selecting  thresholds $\{\Gamma_{j}^k\}$ in order to maximize  $\hat{I}_{0}$. However, for any choice of thresholds,  $\hat{\cS}$ is not (even first-order) asymptotically optimal, since $r \bar{I}_{0}> \hat{I}_{0}$ (see, e.g., \cite{tsi2}).


\subsubsection{Fusion of local CUSUM rules}
Suppose now that  each sensor $k$ communicates at the following times
\begin{equation} \label{meitau}
\tau_{n}^k= \inf\{ t \geq \tau_{n-1}^k: y^k_{t}  \geq c^{k} \} ,
\end{equation}
where  $y^k_{t}:= u^k_{t} - \min_{0 \leq s \leq t} u^k_{s}$ is the local CUSUM statistic  and $c^{k}$ is a fixed, positive threshold. In this way,  
the sensors communicate with the fusion center only  to announce they have detected the change. This requires only  \textit{one-bit} transmissions,
which means that even if the network supports the transmission  of multi-bit messages, this flexibility is not going to be useful.

There are many reasonable fusion center policies that can be based on (\ref{meitau}). For example, the fusion center may raise an alarm  the first time any sensor communicates, i.e., at  $\min_{k} \tau_{1}^k$ (min-CUSUM). This is clearly a one-shot scheme, i.e., it requires transmission of at most one bit from each sensor, and as one would expect it is asymptotically suboptimal (see, e.g., \cite{tart} for the case of i.i.d. observations and \cite{moustdec} for the case of Brownian motions). An alternative possibility is to raise an alarm the first time that all sensors communicate \textit{simultaneously}, i.e., at
$$ 
\cM:= \inf\{ t: y_{t}^k \geq c^{k} , \; \forall \;  k=1, \ldots, K\}.
$$ 
This rule was suggested (although in a different form) by  Mei \cite{mei}, where it was shown that when each $u^{k}$ is a random walk with a finite second moment,  $\cM$ is first-order asymptotically optimal (in particular, $\cJ[\cM]- \cJ[\cS]= \calo(\sqrt{\log \gamma})$), as long as each $c^{k}$ is  proportional to the local KL-number,  $I_{0}^k$. Since the constant of proportionality is determined by $\gamma$, this means that for this decentralized scheme,  contrary to Q-CUSUM, it is not possible to control how often each sensor communicates  with the fusion center. However, by construction, the induced communication activity will be intense only after the  change  has occurred; before the change, a sensor communicates only to report a ``local false alarm'', which is a rare event. Finally, despite its asymptotic optimality, it is known (see, e.g., \cite{mei}, \cite{tart}) that the non-asymptotic performance of $\cM$ can be worse than that of Q-CUSUM when the latter requires transmission of one-bit messages ($b=2$) at every observation time ($r=1$),  especially when $K$ is large.

\section{D-CUSUM}
In this section, we define and analyze the decentralized detection structure that we propose. Thus, we  suggest that each sensor $k$ communicates with the fusion center at the following sequence of $\{\cFt^k\}$-stopping times
\begin{equation} \label{taucd}
\tau_{n}^k := \inf \{ t > \tau_{n-1}^k: u^k_{t} - u^k_{\tau_{n-1}^k} \notin (-\uDi, \oDi) \} , ~ n \in \mathbb{N},
\end{equation}
where $\tau_0^k:=0$ and  $\oDi, \uDi$ are fixed,  positive thresholds. For every $n \in \mathbb{N}$ and $t>0$ we set   
$$\tau^{k}(t):=\tau^{k}_{m_{t}^{k}}, \quad m_{t}^k:= \max\{n \in \mathbb{N}: \tau_{n}^k \leq t\}, \quad \ell_{n}^{k}:= u^k_{\tau_{n}^k}- u^k_{\tau_{n-1}^k},$$
i.e., $m_{t}^k$ is the number of messages that have been transmitted by sensor $k$ up to time $t$, $\tau^{k}(t)$ is the most recent communication time for sensor $k$ at time $t$ and $\ell_{n}^{k}$ is the accumulated log-likelihood ratio at sensor $k$ in the time-interval $[\tau_{n-1}^k, \tau^k_{\tau_{n}^k}]$.

At time $\tau_n^k$, we suggest that sensor $k$ transmits to the fusion center the following message 
\begin{equation} \label{zcdmany}
z_n^k:=\left\{\begin{array}{cl} j , &\text{if} \quad \oei_{j-1} \leq \ell_{n}^{k} -\oDi < \oei_{j}\\
  -j , &\text{if} \quad  -\uei_{j}  < \ell_{n}^{k} + \uDi \leq -\uei_{j-1}
\end{array}\right.
j=1,\ldots,d,
\end{equation} 
where  $\oei_{0}:=\uei_{0}:=0$,  $\oei_{d}:=\uei_{d}:= \infty$,  
$\{\oei_{j}, \uei_{j}\}_{1 \leq j \leq d-1}$ are fixed,  positive threshold and $d$ a positive integer. We will also use the following notation
$$\oDi_{j}:=\oDi + \oei_{j-1} , \quad  \uDi_{j}:=\uDi + \uei_{j-1}, \quad j=1,\ldots,d,$$
which allows us to rewrite (\ref{zcdmany}) as follows
$$
z_n^k=\left\{\begin{array}{cl} j , &\text{if} \quad \oDi_{j} \leq \ell_{n}^{k} < \oDi_{j+1}\\
  -j , &\text{if} \quad  -\uDi_{j+1} <  \ell_{n}^{k}  \leq -\uDi_{j}
\end{array}\right., \quad 
j=1,\ldots,d.
$$ 
When $d=1$, $z_{n}^{k}$ is a one-bit message of the form 
 \begin{equation} \label{zcd}
z_n^k:=\left\{\begin{array}{cl}1 , &\text{if}~  \ell^k_{n}  \geq \oDi\\
-1 , &\text{if}~ \ell^k_{n} \leq -\uDi
\end{array}\right.
\end{equation}
that simply informs the fusion center whether  $\ell^k_{n} \geq \oDi$ or $\ell^k_{n}\leq -\uDi$. When $d \geq 2$, $z_{n}^{k}$ requires the transmission of  
$\lceil \log_{2}(2d) \rceil=1+ \lceil \log_{2} d\rceil$ bits 
and the fusion center also obtains information regarding the size of the overshoot.

The stopping times (\ref{taucd}) and the messages (\ref{zcdmany}) determine the flow of information (\ref{flow}) at the fusion center.
Assuming that the fusion center uses this information and approximates each local log-likelihood ratio $\{u_{t}^{k}\}$ by some statistic $\{\tilde{u}_{t}^{k}\}$,
we suggest  the  following detection rule
\begin{equation} \label{dcusum}
\dcus:= \inf \{ t \geq 0: \tilde{y}_{t}  \geq \tilde{\nu} \}, ~\text{where}~ \tilde{y}_{t}:=\tilde{u}_{t} - \inf_{0 \leq s \leq t} \tilde{u}_{s},\; \tilde{u}_{t}:= \sum_{k=1}^{K} \tilde{u}_{t}^k
\end{equation}
and threshold $\tilde{\nu}$ is defined so that $\Exp_{\infty}[-u_{\dcus}]=\gamma$. The appropriate selection for $\tilde{u}_{t}^{k}$, as well as the design and 
analysis  of the resulting detection rule, is different in discrete and continuous time and, for this reason, we will treat these two setups separately. We will see, however, that the proposed detection structure, that we will call  D-CUSUM, can be designed in order to have strong asymptotic optimality properties in both cases.

\subsection{Continuous-time setup}\label{sec:D-CUSUM-cont}
Suppose that each $\{u^{k}_{t}\}$ is a continuous-time process with continuous paths so that condition (\ref{full}) is satisfied, in which case we have
the following closed-form expressions for $\cJ[S]$ and $\gamma$ in terms of threshold $\nu$ (see, e.g., \cite{moustito},\cite{chro}):  
\begin{align} \label{fapito}
\begin{split}
 \gamma &= \Exp_{\infty}[-u _{\cus}]= \Exp_{\infty}[\langle u \rangle_{\cus}] = e^{\nu}-\nu-1 , \\
 \cJ[\cus] &=  \Exp_{0}[u _{\cus}]= \Exp_{0}[\langle u \rangle_{\cus}]= e^{-\nu}+\nu-1.
\end{split}
\end{align}
Then, each $\ell_{n}^{k}$ is exactly equal to either $\oDi$ or $-\uDi$ and, consequently, at $\tau_{n}^k$ sensor $k$ can transmit  to the fusion center the \textit{exact} value of $\ell_{n}^{k}$ by simply communicating a \textit{one-bit} message of the form (\ref{zcd}).  As a result, the fusion center is able to recover the value of $u^k$ at any time $\tau_{n}^k$, since $u_{\tau_{n}^k}^k = \ell_{1}^{k}+ \ldots+\ell_{n}^{k}$, and a natural approximation for $u_t^k$ at some arbitrary time $t$ is the corresponding most recently reproduced value, i.e., 
\begin{equation} \label{freecd}
 \tilde{u}_{t}^k := u^{k}_{\tau^{k}(t)} =\sum_{n=1}^{m_{t}^{k}} \ell_{n}^{k}.
\end{equation}

The proposed scheme has a number of practical advantages. First of all, the  fusion statistic $\{\tilde{y}_{t}\}$ is piecewise-constant and its value needs to 
be updated only at communication times, according to the following convenient formula:
$$
\tilde{y}_{\tau_{n}^{k}}=(\tilde{y}_{\tau_{n}^{k}\text{-}})^{+} + \oDi\ind{z_n^k=1}-\uDi\ind{z_n^k=-1}.
$$
Compare this with the centralized, continuous-time CUSUM statistic, $\{y_{t}\}$, 
which does not in general admit such a recursion 
and  whose calculation at the fusion center requires high-frequency transmission of ``infinite-bit'' messages from the sensors. 

Moreover, it is possible  to control the communication rate of sensor $k$ by selecting appropriately $\oDi$ and $\uDi$. 
Since $\Exp_{i}[\tau_{n}^k -\tau_{n-1}^k]$, $i=0, \infty$ in general depend on $n$, 
these thresholds can be selected  in order to attain target values for $\Exp_{0}[\ell^k_{n}]$ and  $\Exp_{\infty}[-\ell^k_{n}]$,  
which do not depend on $n$ and are given by   $\Exp_{0}[\ell^k_{n}]=s(\uDi , \oDi)$ and $\Exp_{\infty}[-\ell^k_{n}] =s(\oDi,\uDi)$, where 
$$s(x,y) :=  \frac{-x(e^{y}-1)+ye^{y}(e^{x}-1)}{e^{x+y}-1}.$$
In this way, the specification of $\oDi$ and $\uDi$ simply requires the solution of a (non-linear) system of two equations.

From the previous discussion  it should be clear that  D-CUSUM is much more preferable than the corresponding centralized CUSUM from a practical point of view. 
It turns out that it also has excellent performance characteristics, making any  additional benefit of the optimal centralized CUSUM test negligible relative to its implementation cost. This becomes clear with the following theorem, which provides a non-asymptotic upper bound on the performance loss of the proposed detection structure.   
 
\begin{theorem} \label{prop1}
For any $\gamma$ and $\{\oDi,\uDi\}_{1 \leq k \leq K}$ we have
\begin{equation} \label{order2cd}
\cJ[\dcus] - \cJ[\cus] \leq 4 \,  K \,  \D_{\max} , \quad \text{where} \quad \D_{\max}:= \max_{1 \leq k \leq K} \max \{\oDi, \uDi\}.
\end{equation}
\end{theorem}

\begin{proof}
The proof is presented in Appendix\,\ref{app:A}.
\end{proof}

The bound provided in \eqref{order2cd} implies that for any fixed thresholds $\{\oDi, \uDi\}$ and any
number of sensors $K$, $\cJ[\dcus] - \cJ[\cus]=\calo(1)$  as $\gamma \rightarrow \infty$, i.e., $\dcus$ is \textit{second}-order asymptotically optimal. In the case of a large sensor-network ($K \rightarrow \infty$), this property is preserved only if we have an asymptotically  high rate of communication, specifically if $\D_{\max} \rightarrow 0$ so that $K  \D_{\max} =\calo(1)$.  However, since we want to avoid intense transmission activity, it  is more interesting to see that $\dcus$ remains \textit{first}-order asymptotically optimal  when $K \rightarrow \infty$ and $\D_{\max} \rightarrow \infty$ so that $K \D_{\max}=o(\log \gamma)$. Indeed, from (\ref{fapito}) and (\ref{order2cd}) we have 
$$
\frac{\cJ[\dcus]}{\cJ[\cus]} = 1+ \frac{\cJ[\dcus] - \cJ[\cus]}{\cJ[\cus]} \leq  1+  \frac{4 K \D_{\max}}{e^{-\nu}+\nu-1} 
$$
and our claim now also follows  from  (\ref{fapito}), which implies that  $\nu=\log \gamma+o(1)$.

\subsection{Discrete-time setup}
Suppose now that each $\{u^{k}_{t}\}$ is  a random walk, i.e., the increments $\{u^k_{t}-u_{t-1}^{k}\}_{t \in \mathbb{N}}$ are i.i.d. This implies that each
$(\tau_{n}^{k}-\tau_{n-1}^{k}, z_{n}^{k},\ell_{n}^{k})_{n\in \mathbb{N}}$ is a sequence of independent triplets with the same distribution as 
$(\tau_{1}^{k}, z_{1}^{k},\ell_{1}^{k})$. As a result, thresholds $\oDi$ and $\uDi$ can now be selected in order to attain target values for $\Exp_{i}[\tau_{1}^{k}]$, $i=0,\infty$. However, the main difference with the continuous-time setup is that now each $\ell_{n}^{k}$ is no longer restricted to the binary set $\{\oDi, -\uDi\}$. Thus, it now makes sense to have larger than binary alphabets ($d>1$), in which case
we also need to select thresholds $\{\oei_{j}, \uei_{j}\}_{1  \leq  j \leq d-1}$ (recall that  $\oei_{0}=\uei_{0}:=0$,  $\oei_{d}=\uei_{d}:= \infty$). We suggest the following specification  
\begin{align}
&\Pro_0(\ell^k_{1}-\oDi \geq \oei_j \, | \, \ell_{1}^{k} \geq \oDi)=1-\frac{j}{d} =\Pro_\infty(\ell^k_{1}+\uDi \leq -\uei_j \, | \,\ell_{1}^{k} \leq -\uDi),
\label{eq:levels}
\end{align}
which guarantees that  the overshoot $\ell_{1}^{k} - \oDi$ (resp.  $-(\ell_{1}^{k} + \uDi)$ is equally likely to lie in each interval $[\oei_{j-1},\oei_{j})$  (resp. $(-\uei_{j}, -\uei_{j-1}]$ given that  $\ell_{1}^{k} \geq \oDi$ (resp.  $\ell_{1}^{k} \leq -\uDi$), i.e., 
\begin{align*}
&\Pro_0(\ell^k_{1}-\oDi \in [\oei_{j-1}, \oei_j) \, | \, \ell_{1}^{k} \geq \oDi)=\frac{1}{d}
=\Pro_\infty(\ell^k_{1}+\uDi \in (-\uei_{j}, -\uei_{j-1}] \, | \, \ell_{1}^{k} \leq -\uDi),
\end{align*}
or, equivalently,  $\Pro_0( z_{1}^{k}=j  \, | \, z_{1}^{k}>0 )=1/d
=\Pro_\infty( z_{1}^{k}=-j  \, | \, z_{1}^{k}<0 )$,  for every $1 \leq j \leq d$. 
Clearly, all these thresholds can be easily computed off-line, as their computation requires the simulation of  the pair $(\tau_{1}^{k}, \ell_{1}^{k})$ under both $\Pro_{0}$  and $\Pro_{\infty}$.  
Moreover, in what follows,  we assume that $u_{1}^{k}$ is unbounded and  absolutely continuous with a positive density.  Then,  $\oei_{d-1}, \uei_{d-1} \rightarrow \infty$  as $d \rightarrow \infty$, whereas 
\begin{equation} \label{del}
\epsilon^{k}:= \max_{1 \leq j \leq d-1} \, \{ \oei_{j}-\oei_{j-1} \, , \,  \uei_{j}-\uei_{j-1} \} \rightarrow 0 \quad  \text{as} \quad  d \rightarrow \infty.
\end{equation} 

In order to establish a  \textit{second-order} asymptotic optimality property for $\tilde{\cS}$, as in the continuous-time setup,  
we need a lower bound for the optimal centralized performance $\cJ[\cS]$ up to a constant term as $\gamma \rightarrow \infty$. Moreover, in order to obtain 
the inflicted performance loss as $K \rightarrow \infty$, we need to characterize the growth of this constant term  as $K \rightarrow \infty$. This is done in the following lemma, under a second moment condition on each $u_{1}^{k}$.

\begin{lemma}\label{lem:6}
If $\Exp_{0}[(u_{1}^{k})^{2}]<\infty$ for every $1  \leq k \leq K$, then for any $\gamma$ we have 
\begin{equation}
\cJ[\cS]= \Exp_0[u_{\cS}]\ge\log\gamma- {\mit \Theta}(K).
\end{equation}
\end{lemma}

\begin{proof}
It is well known that the worst case for the optimal centralized CUSUM is when the change occurs at $\tau=0$, which implies the equality in the lemma. The proof of the inequality is presented in Appendix \ref{app:B}.
\end{proof}

If  each sensor $k$ transmitted the exact value of each $\ell_{n}^{k}$ at time $\tau_{n}^{k}$, as in the continuous-time setup,  then we could  approximate $u_{t}^{k}$ by (\ref{freecd}) and we could work in the same way as  Theorem \ref{prop1} to show that $\cJ[\dcus] -\cJ[\cus] = \calo(K \D_{\max})$.  However, this is not possible in a discrete-time setup, since $\ell_{n}^{k}$ cannot be fully recovered at the fusion center when sensor $k$ transmits only a small number of bits at time $\tau_{n}^{k}$. Our main goal in the remainder of the paper is to show that it is actually possible to design D-CUSUM in discrete time so that it is second-order asymptotically optimal even if each sensor transmits a small number of bits (such as 2 or 3) in every communication. In order to do this, we  approximate $u_{t}^{k}$ by 
\begin{equation} \label{tilder}
\tilde{u}_{t}^k := \sum_{n=1}^{m_{t}^{k}} \tilde{\ell}_{n}^{k},
\end{equation}
where  $\tilde{\ell}_{n}^{k}$ is the log-likelihood ratio of $z_n^k$, i.e., 
\begin{align}
 \tilde{\ell}_{n}^{k} &:= \sum_{j=1}^{d} \Bigl[ \oli_{j} \, \ind{z_{n}^k=j} - \uli_{j} \, \ind{z_{n}^k=-j} \Bigr], \quad  \label{ells} \\
 \oli_{j} &:= \log  \frac{  \Pro_{0}(z_1^k=j)}{\Pro_{\infty}(z_1^k=j)}, ~   -\uli_{j} := \log   \frac{\Pro_{0}(z_1^k=-j)}{\Pro_{\infty}(z_1^k=-j)}. 
 \label{lambdas}
\end{align}
The log-likelihood ratios $\{\oli_{j},\uli_{j}\}$ do not admit closed-form expressions, however they can be easily computed via simulation. This is not an easy task if one uses their definition in  (\ref{lambdas}), which requires simulation of rare events, especially when $\oDi, \uDi$ are  large. However, 
we can overcome this problem using the following lemma. 

\begin{lemma} \label{lem:0}
For every  $1\leq j \leq d$, $\oli_{j}=\oDi_{j}+\oRi_{j}$ and $\uli_{j}=\uDi_{j}+\uRi_{j}$, where 
\begin{align} \label{import}
\begin{split}
\oRi_{j} &:= - \log \Exp_{0}[ e^{-(\ell^k_{1}- \oDi_{j})} \, | \, z_{1}^{k}=j ] >0 , \\
\uRi_{j} &:= - \log \Exp_{\infty}[ e^{\ell^k_{1}+ \uDi_{j}} \, | \, z_{1}^{k}=-j]>0.
\end{split}
\end{align}
Moreover, for every  $1\leq j \leq d-1$, $\oRi_{j}, \uRi_{j} \leq \epsilon^{k}$ and if, additionally, $\Exp_{i}[(u_{1}^{k})^{2}]<\infty$, $i=0, \infty$, then
\begin{align} \label{import2}
\begin{split}
\oRi_{d} &\leq  \Exp_{0}[ \ell^k_{1}- \oDi_{d} \, | \, z_{1}^{k}=d ] \leq \Theta(1) \; d \; \Exp_{0}[(u_{1}^{k})^{2} \ind{u_{1}^{k} \geq \oei_{d-1}}], \\
\uRi_{d} &\leq  \Exp_{\infty}[ -(\ell^k_{1} + \uDi_{d}) \, | \, z_{1}^{k}=-d ] \leq  \Theta(1)  \; d \; \Exp_{\infty}[(u_{1}^{k})^{2} \ind{-u_{1}^{k} \geq \uei_{d-1}}] ,
\end{split}
\end{align}
where $\Theta(1)$ is a term that does not depend on $d$ and is bounded  from above and below as $\oDi, \uDi \rightarrow \infty$.
\end{lemma}

\begin{proof}
The proof can be found in Appendix\,\ref{app:D}.
\end{proof}

Lemma \ref{lem:0} shows that, similarly to the thresholds $\{\oDi_{j},\uDi_{j}\}$ and $\{\oei_{j},\uei_{j}\}$, the 
log-likelihood ratios $\{\oli_{j},\uli_{j}\}$ can be computed off-line and efficiently if we simulate $(\tau_{1}^{k}, \ell_{1}^{k})$ under $\Pro_{0}$ and $\Pro_{\infty}$. Moreover, Lemma \ref{lem:0} shows that defining $\tilde{\ell}_{n}^{k}$ as the log-likelihood ratio of $z_{n}^{k}$ accounts for the unobserved overshoots at the fusion center. Specifically, when the fusion center receives message $z_{n}^{k}=j$ for some $j=1, \ldots, d$, it understands that $\ell_{n}^{k} \in [\oDi_{j}, \oDi_{j+1})$ and it approximates $\ell_{n}^{k}$ by $\oDi_{j}+ \oRi_{j}$; in other words, the fusion center approximates the random overshoot $\ell_{n}^{k}-\oDi_{j}$ that it does not observe by the constant $\oRi_{j}$, which is clearly an $\calo(1)$ term as $\oDi, \uDi \rightarrow \infty$.  

The following lemma is important for quantifying the additional detection delay due to using $\tilde{\ell}_{n}^{k}$ instead of the actual value of $\ell_{n}^{k}$ in (\ref{tilder}).

\begin{lemma} \label{lem:1}
If $\Exp_{i}[(u_{1}^{k})^{2}]<\infty$, $i=0, \infty$, then $\Exp_{0}[\ell^k_1- \tilde{\ell}^k_1 ] \leq 2\theta^{k}$, where
\begin{equation} \label{thetak}
\theta^{k}:= \epsilon^{k} + \Theta(1) \, \Exp_{0}[(u_{1}^{k})^{2} \ind{u_{1}^{k} \geq \oei_{d-1}}]+ \Theta(1) \, \Exp_{\infty}[(u_{1}^{k})^{2} \ind{-u_{1}^{k} \geq \uei_{d-1}}]
\end{equation}
and $\Theta(1)$ is a term that does not depend on $d$ and is bounded  from above and below as $\oDi, \uDi \rightarrow \infty$. Moreover, $\theta^{k} \rightarrow 0$ as $d \rightarrow \infty$.
\end{lemma}

\begin{proof}
The proof of this lemma can be found in Appendix\,\ref{app:D}.
\end{proof}

Note that an alternative approach would have been to define $\tilde{\ell}_{n}^{k}$ as in (\ref{ells}), 
but with  $\oli_{j}$ and $\uli_{j}$ replaced by $\oDi_{j}$ and $\uDi_{j}$, respectively. In this way, the overshoots are simply ignored by the fusion center.  However, the main reason for defining $\tilde{\ell}_{n}^{k}$ as the log-likelihood ratio of $z_{n}^{k}$ is that it allows us to prove the following lemma, which connects threshold $\tilde{\nu}$ with the false alarm period $\gamma$ and  plays a crucial role in establishing the (second-order) asymptotic optimality of the resulting detection rule. 

\begin{lemma} \label{lem3}
For any $\gamma>0$ we have  
$\tilde{\nu} \leq  \log \gamma - \log(\bar{I}_\infty)$, 
thus, $\tilde{\nu} =\log \gamma  +\Theta(1)$ as $\gamma \rightarrow \infty$.
\end{lemma}

\begin{proof}
The proof is presented in Appendix \ref{app:C}.
\end{proof}

It is possible to prove Lemma \ref{lem3} and, consequently, to establish the asymptotic optimality of $\tilde{S}$ if  $\tilde{\ell}_{n}^{k}$ is defined as 
the  log-likelihood ratio of the pair $(\tau_{n}^k-\tau_{n-1}^{k}, z_{n}^{k})$, and not  only of $z_{n}^{k}$. Unfortunately,  the distribution of  $\tau_{1}^k$  is typically intractable, thus, the resulting rule could not be implemented in practice.

We are now ready to state the discrete-time analogue of Theorem \ref{prop1}. For simplicity, we assume that  communication rates, before and after the change, are of the same order of magnitude for all sensors,  i.e., there is a quantity $\D$ so that  $\oDi,\uDi=\Theta(\D)$ as  $\D,\oDi,\uDi \rightarrow \infty$ for all $1\leq k \leq K$. Moreover, we set $\theta:=\max_{1 \leq k \leq K} \theta^{k}$.

\begin{theorem} \label{th:2} 
If $\Exp_{0}[(u_{1}^{k})^{2}]<\infty$ for every $1  \leq k \leq K$, then 
\begin{equation} \label{order1cd}
\cJ[\dcus] -\cJ[\cus] \leq \frac{\theta}{\Theta(\D)} \, \log\gamma  + K \, \Theta(\D).
\end{equation}
\end{theorem}

\begin{proof}
For the optimum CUSUM $\cS$, it is well known that $\cJ[\cS]=\Exp_0[u_{\cS}]$. In order to see that this is also the case for D-CUSUM, i.e., 
$\cJ[\tilde{\cS}]=\Exp_0[u_{\tilde{\cS}}]$, from the nonnegativity of the KL-divergence it is clear that it suffices to show that 
$\tilde{S}\ind{\tilde{S} \geq \tau}= \inf\{t \geq\tau: \tilde{y}_{t} \geq \tilde{\nu}\}$ is \textit{pathwise} decreasing with respect to $\tilde{y}_\tau$, 
or equivalently that the process $\{\tilde{y}_t, t>\tau\}$ is \textit{pathwise} increasing with respect to $\tilde{y}_\tau$. Indeed, if we denote by ($\tau_{n})$ the sequence of times at which there is a communication from at least one sensor, then  
$$
\tilde{y}_{\tau_n}=(\tilde{y}_{\tau_n-})^++\omega_{\tau_n}
$$
where $\omega_{\tau_n}$ is information coming from the sensors that communicate at time $\tau_n$ and is clearly independent from the past. This implies that $\tilde{y}_t$ will be increasing in $(\tilde{y}_\tau)^+$ for any $t \geq \tau$ and our claim follows because the smallest value of the latter quantity is 0.

Based on the above, we can write 
\begin{equation}
\cJ[\tilde{\cS}]-\cJ[\cS]=\Exp_0[u_{\tilde{\cS}}]-\Exp_0[u_{\cS}]=\Exp_0[u_{\tilde{\cS}}-\tilde{u}_{\tilde{\cS}}]
+\Exp_0[\tilde{u}_{\tilde{\cS}}]-\Exp_0[u_{\cS}].
\label{eq:3-part}
\end{equation}
From Lemma\,\ref{lem:8} we have that $\Exp_0[\tilde{u}_{\tilde{\cS}}] \leq \log \gamma +K \Theta(\D)$ and 
$$\Exp_0[u_{\tilde{\cS}}-\tilde{u}_{\tilde{\cS}}] \leq K \Theta (\D)+ \theta \, \frac{\log \gamma}{\Theta(\D)}.$$
Applying these inequalities and Lemma\,\ref{lem:6} to \eqref{eq:3-part}, we obtain the desired result. Lemma \,\ref{lem:8}, as well as some additional auxiliary results, are stated and proved in  Appendix \ref{app:E}.
\end{proof}

The main consequence  of Theorem \ref{th:2} is that D-CUSUM is second-order asymptotically optimal, i.e.,  $\cJ[\dcus] -\cJ[\cus]= \calo(1)$, when $K=\calo(1)$,
$\D=\calo(1)$ and $\theta \rightarrow 0$ so that $\theta \log\gamma=\calo(1)$ as $\gamma \rightarrow \infty$.  We have seen in Lemma \ref{lem:1} that $\theta \rightarrow 0$ as $d \rightarrow \infty$. If, in particular,  $\theta= \calo(1/d^{\alpha})$, where $\alpha$ is some positive constant, then the above analysis implies that $d$ may go to infinity with a rate as low as $\calo((\log \gamma)^{1/\alpha})$ and, as a result, the required number of bits per transmission, $1+ \lceil \log_{2} d \rceil$, can be of an order as low as 
$\calo(\frac{1}{\alpha} \log \log \gamma)$. This means that second-order asymptotic optimality is achieved in practice with a very low number of bits per transmission, a conclusion that will also be supported by some simulation experiments in the end of this section.

As in continuous time, second-order asymptotic  optimality is not preserved with an asymptotically low-rate of communication ($\D \rightarrow \infty$).
However, from Theorem\,\ref{th:2}  and Lemma\,\ref{lem:6} we have
\begin{equation} \label{asy2}
\frac{\cJ[\dcus]}{\cJ[\cus]} =1+ \frac{\cJ[\dcus] -\cJ[\cus]}{\cJ[\cus]} \leq 1+ \frac{ \frac{\theta}{\Theta(\D)} +  \frac{K \Theta(\D)}{\log\gamma}}{1- \frac{\Theta(K)}{\log \gamma}},
\end{equation}
which implies that D-CUSUM is first-order asymptotically optimal,  i.e.,  $\cJ[\dcus]/\cJ[\cus] \rightarrow 1$, when $\D \rightarrow \infty$ so that  $K \D=o(\log \gamma)$. In this context, the performance of D-CUSUM is optimized  when $\D, \theta, K$ are selected so that the two terms in the upper bound of \eqref{order1cd} are of the same order magnitude. This happens when $\D=\Theta(\sqrt{\theta \log\gamma /K})$, in which case $\cJ[\dcus] -\cJ[\cus]= \calo(\sqrt{K \, \theta  \, \log\gamma})$. 

We should emphasize that in the case of a binary alphabet ($d=1$), where $\theta$ is bounded away from 0 (i.e., $\theta=\Theta(1)$),  
first-order asymptotic optimality cannot be achieved with a fixed rate of communication, i.e., when $\D=\calo(1)$ as $\gamma \rightarrow \infty$.  This may seem counterintuitive at first, however it is quite reasonable since a high rate of communication leads to fast accumulation of quantization error. Nevertheless, 
this source of error can be suppressed if we have a sufficiently large alphabet size that allows us to quantize the overshoots. This explains why first-order asymptotic optimality can be achieved even with $\D=\calo(1)$ when $\theta \rightarrow 0$. 


We conclude that, either with a high or a low communication rate, the performance of D-CUSUM is improved with a larger than binary alphabet $(d>1)$, but in practice  a small value of $d$ should be sufficient. In order to elaborate more on this point, let us note that the statistical behavior of the overshoots depends on the  parameter $\D$, which controls the average period of communication in the sensors. However, this dependence is only minor since the distribution of the overshoots converges to some limiting distribution as $\D$ becomes large. In other words, quantizing the overshoots is like quantizing a random variable with (almost) fixed statistics. Consequently, the mean square quantization error, or any other similar quality measure,  will be (almost) independent from $\D$ for fixed number of bits.

On the contrary,  for the classical quantization scheme \eqref{zdd}, employed by Q-CUSUM, quantization is applied on the value of each $u^k_{nr}-u^k_{(n-1)r}$, where $r$ denotes  the fixed corresponding period. It is very simple to realize that for fixed number of bits, if we increase the period $r$, the mean square quantization error will \textit{increase}, since the difference $u^k_{nr}-u^k_{(n-1)r}$ will involve a larger sum of i.i.d.~random variables. This becomes particularly obvious when these random variables are bounded, in which case the support of the sum increases linearly with $r$ and we are asked, with the same number of bits, to quantize a larger range of values. This suggests that if we want to communicate with the fusion center at a smaller rate and preserve the same number of bits, this will inflict larger quantization errors and therefore additional performance degradation for Q-CUSUM. As we mentioned above, 
this is not the case with the quantization scheme we adopt for D-CUSUM, since increasing $\D$ (to reduce the communication rate) leaves the mean square quantization error almost intact.

Let us now illustrate these conclusions with a simulation study. Specifically, suppose that each sensor $k$ takes independent, normally distributed observations with variance $1$ and mean that changes from $0$ to $\mu$,  i.e., $\xi_{t}^k \sim \cN(0,1)$ when $t \leq \tau$ and $\xi_{t}^k \sim \cN(\mu,1)$ when $t > \tau$. Then, for every $t \in \mathbb{N}$ we have $u_{t}^k-u_{t-1}^{k} =  \mu \, \xi_{t}^k-  \mu^{2}/2$. 
We assume that $\oDi=\uDi=\Di$ and for every $j=1, \ldots, d-1$ we set $\oei_{j}=\uei_{j}=\ei_{j}$ and, consequently, we have $\oli_{j}=\uli_{j}=\li_{j}$. Moreover, we assume that each $\D^{k}$ is chosen so that $\Exp_{0}[\tau_{1}^{k}]=r$. In  Table\,\ref{tab:1}  we present the  values of these parameters when the number of transmitted bits per message is $d=1$ or $d=2$, the communication period is $r=3$ or $r=6$ and  $\mu=1$.

\begin{table}[!h]
	\centering
		\caption{Thresholds and Log-Likelihood Ratios}
\begin{tabular}{cc}
		\begin{tabular}{|c||c|c|c|c|} \hline
	   	                    & $\Di_{1}$  & $\li_{1}$            \\ \hline \hline
		$r=3$, $\mu=1$    & 1.287    & 1.87    \\ \hline
	  $r=6$, $\mu=1$    & 2.54    & 3.12    \\ \hline
	\end{tabular}

&

		\begin{tabular}{|c||c|c|c|c|}\hline
	   	                    & $\Di_{1}$ & $\Di_{2}$ & $\li_{1}$ & $\li_{2}$            \\ \hline  \hline
		$r=3$, $\mu=1$    & 1.287    & 1.87      & 1.54     & 2.94    \\ \hline
	  $r=6$, $\mu=1$    & 2.54    & 3.12      & 2.80     & 3.62     \\ \hline
	\end{tabular}\\
	&\\
(a) $\;d=1$ &	(b) $\;d=2$
	
	\end{tabular}
	\label{tab:1}

	\end{table}

Our goal is to compare  D-CUSUM $\tilde{\cS}$ with  Q-CUSUM $\hat{\cS}$, which was defined in (\ref{eq:cus2}), when both rules use the 
same resources, i.e., the same  number of bits per communication and the same (average) rate of communication. 
Note that such a fair comparison is not possible with decentralized rules that do not explicitly control their transmission rate. 
Of course, the ultimate benchmark is the centralized CUSUM test, which requires transmission of the observation of each sensor at every time $t$.

\begin{figure}[!ht]
\centering
\includegraphics{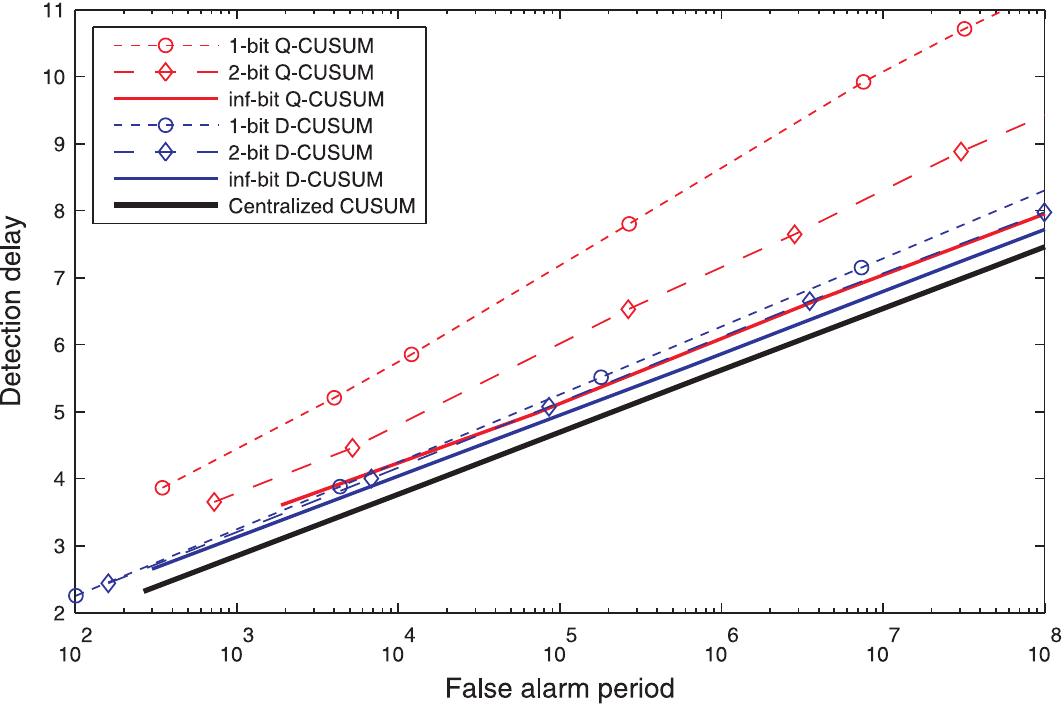} 
\caption{Case of $K=5$ sensors with communication period $r=3$.}
\label{fig:1}
\vskip0.5cm
\centering
\includegraphics{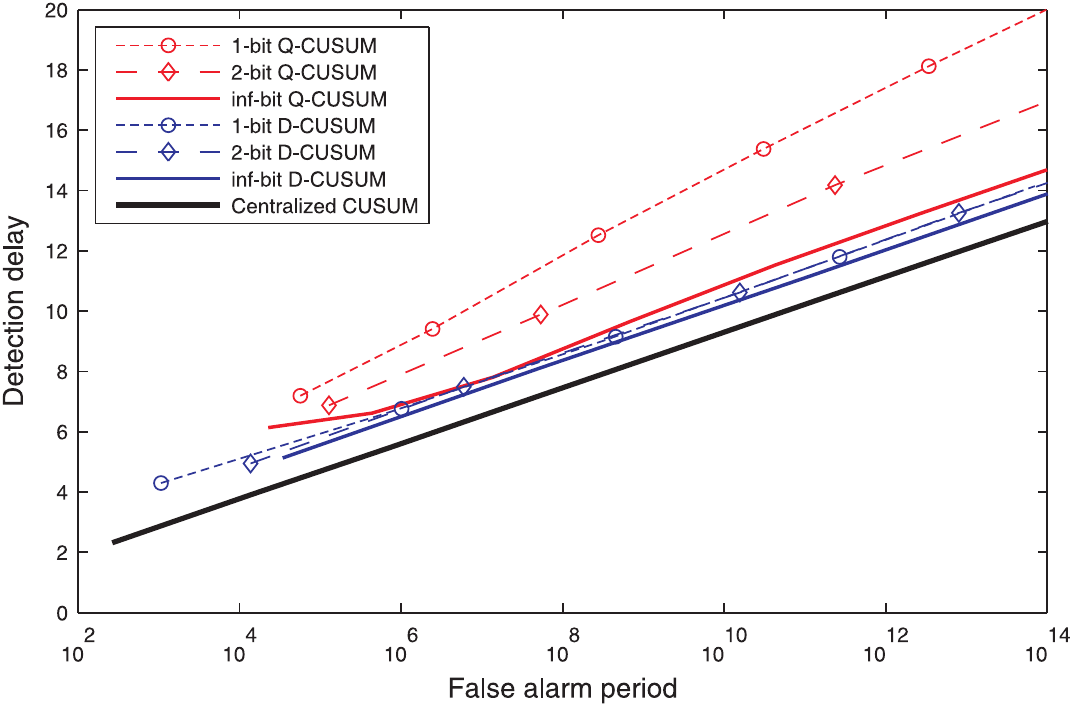} 
\caption{Case of $K=5$ sensors with communication period $r=6$.}
\label{fig:2}
\end{figure}

Fig.\,\ref{fig:1} and Fig.\,\ref{fig:2} depict the main results of our simulations. First of all, we observe that in all cases the operating characteristic curve of D-CUSUM $\dcus$ is essentially parallel to that of the optimal centralized CUSUM, $\cS$. This is exactly the \textit{second-order} asymptotic optimality that we established theoretically. On the contrary, the operating characteristic curve of Q-CUSUM $\hat{\cS}$ diverges as $\gamma$ increases,  as expected, 
since this  not an asymptotically optimal scheme (even of first order). 

Of course, when an ``infinite-bit'' message is transmitted at each communication time, Q-CUSUM corresponds to the centralized CUSUM with period $r$ and its
operating characteristic curve is parallel  to the optimal one. However, what is really interesting is that D-CUSUM with one-bit or two-bit transmissions 
is either very close or even outperforms this  \textit{infinite-bit} Q-CUSUM.

Finally, we should also note that when the average communication period is small ($r=3$), there is a considerable improvement in D-CUSUM when using two, instead of one, bits per transmission (see  Fig.\,\ref{fig:1}). On the other hand, when the average communication period is large ($r=6$), we do not observe similar performance gains for D-CUSUM by having the sensors transmit additional bits to the fusion center (see Fig.\,\ref{fig:2}).

\section{Conclusions} 
The main contribution of this paper is a novel decentralized sequential detection rule, that we called D-CUSUM, according to which each sensor communicates with the fusion center at two-sided exit times of its local log-likelihood ratio and the fusion center  uses in parallel a CUSUM-like rule in order to detect the change. We showed that the performance loss of D-CUSUM with respect to the optimal centralized CUSUM is bounded as the rate of false alarms goes to 0 (\textit{second order} asymptotic optimality). Moreover, we showed that  its first-order asymptotic optimality is preserved even with an asymptotically low communication rate and  large number of sensors. We illustrated  these properties with simulation experiments, which also showed that D-CUSUM performs significantly better than a CUSUM-based, decentralized detection rule that requires communication at deterministic times. 

We assumed throughout the paper that observations from different sensors are independent, an assumption which is not needed for the optimality of the centralized CUSUM test, but is universal in the decentralized literature. This assumption is necessary both for the design and the analysis of D-CUSUM in discrete time, however it is possible to remove it in continuous time, at least when the sensors observe \textit{correlated} Brownian motions. Indeed, going over the proof of Theorem \ref{prop1} in Appendix\,\ref{app:A}, we realize  that this assumption is needed only to the extent that it guarantees a decomposition of the form $u_{t}= \sum_{k=1}^{K} u_{t}^k$, where $\{u^k_{t}\}$ is an $\cFt^k$-adapted process with continuous paths. That is, we did not use explicitly the fact that $\{u^k_{t}\}$ is the local log-likelihood ratio at sensor $k$. This implies that Theorem \ref{prop1} will remain valid even for sensors with correlated dynamics, as long as such a decomposition is possible. This is indeed the case when the sensors observe correlated Brownian motions before and after the change, i.e., for every $1 \leq k \leq K$ it is
$$ \xi^k_{t}=   \sum_{j=1}^{K} \sigma_{kj} W^{j}_{t} + \ind{t > \tau} \, \mu^k t, ~ t \geq 0,$$
where  $(W^{1}, \ldots,W^{K})$ is a standard $K$-dimensional Wiener process, $\mu=[\mu^{1}, \ldots, \mu^{K}]'$ a $K$-dimensional real vector 
and $\sigma:=[\sigma_{ij}]$ a square matrix of dimension $K$ so that the diffusion coefficient matrix ${\mit\Sigma}= \sigma \sigma'$ is invertible. 
Then, we can write $u_{t}= \sum_{k=1}^{K} [ b^k \, \xi_{t}^k - 0.5 \, \mu^k \, b^k \, t]$,
where $b=[b^1,\ldots,b^K]'= {\mit\Sigma}^{-1} \mu$, and Theorem\,\ref{prop1} remains valid as long as we define $u^{k}_{t}$ in (\ref{taucd})  
not as the local log-likelihood ratio $\mu^k \, \xi_{t}^k - 0.5 \, (\mu^k)^{2}  \, t$, but as  $b^k  \xi_{t}^k - 0.5 \, \mu^k b^k t$. However, 
it  remains an open problem to establish asymptotically optimal, decentralized detection rules for more general continuous-time models, and of course in the i.i.d. setup, when the sensor observations are correlated.

\appendix

\section{}\label{app:A}
In this Appendix, we focus on the continuous-time setup of Subsection \ref{sec:D-CUSUM-cont} and we note that 
$$\Exp_{0}[u_{T}]= \Exp_{0}[\langle u \rangle_{T}] ,~  \Exp_{\infty}[-u_{T}]= \Exp_{\infty}[\langle u \rangle_{T}]$$ 
for any stopping time $T$ for which the above quantities are finite. Moreover, for any $x>0$ we use the following notation 
$$\cus_{x}= \inf \{ t \geq 0: y_{t}  \geq x  \}  ,~  \dcus_{x}= \inf \{ t \geq 0: \tilde{y}_{t} \geq x \}.$$
Then, thresholds $\nu$ and $\tilde{\nu}$  are  chosen so that $\Exp_{\infty}[-u_{\cus_{\nu}}]= \gamma=\Exp_{\infty}[-u_{\dcus_{\tilde{\nu}}}]$, or equivalently,
\begin{equation} \label{lop}
\Exp_{\infty}[\langle u \rangle _{\cus_{\nu}}]= \gamma=\Exp_{\infty}[\langle u \rangle _{\dcus_{\tilde{\nu}}}].
\end{equation}
 The proof of Theorem\,\ref{prop1} is based on the following lemma, for which we set $C:=K \D_{\max}$, where $\D_{\max}:=\max_{1 \leq k \leq K} \max\{\oDi, \uDi\}$.

\begin{lemma} \label{lem}
For any $\gamma>0$ \\ 
(i) $\cS_{\tilde{\nu} -2 C} \leq \dcus_{\tilde{\nu}} \leq \cS_{\tilde{\nu}+ 2 C}$ $\Pro_{0}, \Pro_{\infty}$-\text{a.s.} \quad 
(ii) $|\nu-\tilde{\nu}| \leq 2 C$.
\end{lemma}

\begin{proof}
For any $t>0$, from (\ref{taucd}) and (\ref{freecd}) it is clear that  for every $1 \leq k \leq K$
$$|u_{t}^{k}- \tilde{u}_{t}^{k}| \leq \max\{\oDi, \uDi\}\leq \D_{\max}.$$  
Then, summing over $k$ we obtain  $|u_{t}- \tilde{u}_{t}| \leq K \D_{\max}= C$ and, consequently, $|m_{t}- \tilde{m}_{t}| \leq C$, where
$m_{t}:=\inf_{ 0 \leq s \leq t} u_{s}$ and $\tilde{m}_{t}:=\inf_{ 0 \leq s \leq t} \tilde{u}_{s}$. 
Therefore, from the definition of $y_{t}$ and $\tilde{y}_t$ we have
$$|y_{t}-\tilde{y}_{t}|\leq |u_{t}-\tilde{u}_{t}|+ |m_{t}-\tilde{m}_{t}| \leq 2C,$$
which implies (i).  
From (i) and the fact that $\langle u \rangle$ is an increasing process we have
$$
\Exp_{\infty}[\langle u\rangle_{{\cal{S}}_{\tilde{\nu}-2 C}}] \leq \Exp_{\infty}[\langle u\rangle_{\dcus_{\tilde{\nu}}}] \leq \Exp_{\infty}[\langle u\rangle_{{\cal{S}}_{\tilde{\nu}+2 C}}].
$$
From the last inequality and (\ref{lop}) we obtain 
$$\Exp_{\infty}[\langle u\rangle_{{\cal{S}}_{\tilde{\nu}-2 C}}] \leq  \Exp_{\infty}[\langle u\rangle_{\cus_\nu}] \leq \Exp_{\infty}[\langle u\rangle_{{\cal{S}}_{\tilde{\nu}+2 C}}].$$
Let us now recall (\ref{fapito}) and define the function 
$$\psi(x)  := \Exp_{\infty}[-u_{{\cal{S}}_{x}}]=\Exp_{\infty}[\langle u\rangle_{{\cal{S}}_{x}}] = e^{x}-x-1, \; x>0.$$ 
Then, the last pair of inequalities takes the form $\psi(\tilde{\nu}-2 C) \leq \psi(\nu) \leq  \psi(\tilde{\nu}+2 C)$ and (ii) then 
follows from the fact that $\psi$ is strictly increasing.
\end{proof}

\begin{proof}[Proof of Theorem\,\ref{prop1}]
The proof is a direct consequence of Lemma \ref{lem}(i) and (\ref{fapito}). Indeed, 
\begin{align*}
\cJ[\dcus_{\tilde{\nu}}]- \cJ[\cus_\nu] &\leq \cJ[\cS_{\tilde{\nu}+ 2 C}]- \cJ[\cus_\nu]  
     =  ( e^{-\tilde{\nu} - 2 C} -e^{-\nu})  + (\tilde{\nu} -\nu) +2 C \leq  4 C,
\end{align*}
where the first inequality follows from the nonnegativity of KL-divergences and the fact that $\cS_{\tilde{\nu}+ 2 C} \geq \cus_\nu$,  
the equality is due to the second relationship in (\ref{fapito}) and the second inequality due to the fact that $|\nu-\tilde{\nu}| \leq 2 C$.
\end{proof}

\section{}\label{app:B} 

\begin{proof}[Proof of Lemma \ref{lem:6}]
Let us first define for any $r\ge0$ the stopping times 
$$T_r^+=\inf\{t>0:u_t\geq r\},\quad T_r^-=\inf\{t>0:-u_t\geq r\}.$$
Due to the representation of the CUSUM stopping time as a repeated SPRT with thresholds 0 and $\nu$, we have the following well-known formula (see for example Siegmund, \cite[Page 25]{seigmund}) for its expectation under $\Pro_{0}$ and $\Pro_{\infty}$
 \begin{equation}
\Exp_i[u_{\cS}]=\frac{\Exp_i[u_{\cT}]}{\Pro_i(u_{\cT}\geq\nu)},~i=0,\infty,
\label{eq:APP1}
\end{equation}
where $\cT=\min\{T_0^-,T_\nu^+\}$ is the SPRT stopping time with boundaries 0 and $\nu$. Using \eqref{eq:APP1} for $i=0$, we can now write
\begin{align} \label{eq:APP-10}
\begin{split}
\Exp_0[u_{\cS}]  &= \frac{\Exp_0[u_{\cT}\ind{u_{\cT}\ge\nu}] + \Exp_0[u_{\cT} \ind{u_{\cT}\leq0}]}{\Pro_0(u_{\cT}\geq\nu)} \\
&\geq \nu - \frac{\Exp_0[(-u_{\cT})\ind{u_{\cT}\leq0}]}{\Pro_0(u_{\cT}\geq\nu)}. 
\end{split}
\end{align}
We start with the numerator and with a change of measure we have
\begin{equation}  \label{eq:APP-22}
\Exp_0[-u_{\cT} \ind{u_{\cT}\leq0}] = \Exp_\infty[e^{u_{\cT}} (-u_{\cT}) \ind{u_{\cT}\leq0}] \leq  \Exp_\infty[-u_{\cT} \ind{u_{\cT}\leq0}].
\end{equation}
We can now strengthen this inequality as follows: 
\begin{align} 
\begin{split}
\Exp_\infty[-u_{\cT} \ind{u_{\cT}\leq0}] &= \Exp_\infty[ -u_{\cT_{0}^{-}} \ind{\cT_{0}^{-} \leq \cT_{\nu}^{+}}] \leq \Exp_\infty[-u_{\cT_{0}^{-}}] \\
&\leq  \sup_{r \geq 0} \Exp_\infty[-u_{\cT_{r}^{-}}-r] \leq  \frac{\Exp_\infty[(u_1)^2]}{\Exp_\infty[-u_1]} \\
&\leq \frac{\sum_{k=1}^K\Exp_\infty[(u_1^k-I_\infty^k)^2]+(\sum_{k=1}^KI_\infty^k)^2}{\sum_{k=1}^KI^k_\infty} \\
&=  \frac{\bar{\sigma}_\infty^2}{\bar{I}_\infty} + K\bar{I}_\infty, 
\end{split}
\label{eq:APP-20}
\end{align}
where $\bar{I}_i=\frac{1}{K}\sum_{k=1}^K I_i^k$ is the average, over all sensors, of the Kullback-Leibler information numbers and 
$\bar{\sigma}_i^2:=\frac{1}{K}\sum_{k=1}^K\text{Var}_{i}\{u_1^k\}$ the average, over all sensors, of the variances of the local likelihood ratios $u_1^k$, under the probability measure $\Pro_i,~i=0,\infty$. The second inequality in the second  line in \eqref{eq:APP-20}  follows from Lorden's \cite{lorden} upper bound for the average overshoot, strengthened by observing that $(u_1^-)^2\leq(u_1)^2$. 

Furthermore, for the denominator in (\ref{eq:APP-10}) we have
\begin{align}
\begin{split}
\Pro_0(u_{\cT}\geq\nu) &=\Pro_0(T^+_\nu<T_0^-) \geq\Pro_0(T_0^-=\infty)=\frac{1}{\Exp_0[T_{0}^+]}=\frac{K\bar{I}_0}{\Exp_0[u_{T_{0}^+}]} \\
&\geq\frac{K\bar{I}_0}{\sup_{r\geq0}\Exp_0[u_{T^+_r}-r]} \geq\frac{(K\bar{I}_0)^2}{K\bar{\sigma}_0^2+(K\bar{I}_0)^2}=\frac{\bar{I}_0^2}{K^{-1}\bar{\sigma}_0^2+\bar{I}_0^2}\\
&\geq\frac{\bar{I}_0^2}{\bar{\sigma}_0^2+\bar{I}_0^2}.
\end{split}
\label{eq:APP-40}
\end{align}
The second equality in the first line is a classical result of random walk theory (see for example Siegmund \cite[Corollary~8.39, Page 173]{seigmund}), 
whereas the third equality in the first line is an application of Wald's  identity. The  second inequality in the second line is again the upper bound provided by Lorden \cite{lorden} for the overshoot, while the last inequality is true because $K\geq1$.

From (\ref{eq:APP-22}), (\ref{eq:APP-20}) and (\ref{eq:APP-40})  we obtain 
$$
\frac{\Exp_0[(-u_{\cT})\ind{u_{\cT}\leq0}]}{\Pro_0(u_{\cT}\geq\nu)} \leq \frac{\bar{\sigma}_\infty^2 + K(\bar{I}_\infty)^{2}}{\bar{I}_\infty} \, 
\frac{\bar{\sigma}_0^2+\bar{I}_0^2}{\bar{I}_0^2} =  \Theta(K)
$$
and consequently  from \eqref{eq:APP-10} it follows that $\Exp_0[u_{\cS}]  \geq \nu-\Theta(K)$. It remains to find a  lower bound for $\gamma$ in terms of $\nu$.  From the false alarm constraint and \eqref{eq:APP1} we have
\begin{equation} 
\gamma=\Exp_\infty[-u_{\cS}]=\frac{\Exp_\infty[-u_{\cT}]}{\Pro_\infty(u_{\cT}\geq\nu)}.
\label{eq:APP-2}
\end{equation}
For the expectation in the numerator, we can obtain the following upper bound
\begin{align}
\Exp_\infty[-u_{\cT}]&=\Exp_\infty[-u_{\cT} \ind{u_{\cT}\leq0}]+\Exp_\infty[-u_{\cT} \ind{u_{\cT}\geq\nu}] \nonumber \\
&\leq \Exp_\infty[-u_{\cT} \ind{u_{\cT}\leq0}] \leq \frac{\bar{\sigma}_\infty^2}{\bar{I}_\infty}+K\bar{I}_\infty,
\label{eq:APP-2.1}
\end{align}
where the final inequality follows from (\ref{eq:APP-20}). In order to obtain a lower bound for the probability $\Pro_\infty(u_{\cT}\geq\nu)$ in the denominator 
we start with a change of measure, thus
\begin{equation} \label{lbo}
\Pro_\infty(u_{\cT}\geq\nu)=\Exp_0[e^{-u_{\cT}}\ind{u_{\cT}\geq\nu}]=\Exp_0[e^{-u_{\cT}}|u_{\cT}\geq\nu]\Pro_0(u_{\cT}\geq\nu).
\end{equation}
Then, with an application of the conditional Jensen inequality we have
\begin{align} \label{lb}
\begin{split}
\Exp_0[e^{-u_{\cT}}|u_{\cT}\geq\nu] &\geq \exp(-\Exp_0[u_{\cT}|u_{\cT}\geq\nu]) \\
&\geq \exp\left(-\nu-\frac{\Exp_0[(u_{\cT}-\nu)\ind{u_{\cT}\geq\nu}]}{\Pro_0(u_{\cT}\geq\nu)}\right) \\
&\geq \exp\left(-\nu-\frac{\sup_{r\ge0}\Exp_0[u_{T^+_r}-r]}{\Pro_0(u_{\cT}\geq\nu)}\right) \\
&\geq \exp\left(-\nu-\frac{\frac{\bar{\sigma}_0^2}{\bar{I}_0}+K\bar{I}_0}{\Pro_0(u_{\cT}\geq\nu)}\right).
\end{split}
\end{align}
where in the last inequality we have used, again, Lorden's \cite{lorden} upper bound for the maximal average overshoot. Combining (\ref{lbo}) and (\ref{lb}) we obtain
\begin{align}
\begin{split}
\Pro_\infty(u_{\cT}\geq\nu) &\geq \exp\left(-\nu-\frac{\frac{\bar{\sigma}_0^2}{\bar{I}_0}+K\bar{I}_0}{\Pro_0(u_{\cT}\geq\nu)}\right) \, \Pro_0(u_{\cT}\geq\nu) \\
                             &\geq \exp\left(-\nu-\frac{\frac{\bar{\sigma}_0^2}{\bar{I}_0}+K\bar{I}_0}{\frac{\bar{I}_0^2}{\bar{\sigma}_0^2+\bar{I}_0^2}}\right) \, \frac{\bar{I}_0^2}{\bar{\sigma}_0^2+\bar{I}_0^2},
                             \end{split}
                             \label{eq:APP-2.9}
                             \end{align}
where the second inequality follows from \eqref{eq:APP-40}. Then, from (\ref{eq:APP-2}), (\ref{eq:APP-2.1}) and (\ref{eq:APP-2.9}) we have
$$\gamma \leq \Bigl(\frac{\bar{\sigma}_\infty^2}{\bar{I}_\infty}+K\bar{I}_\infty \Bigr)\,  \exp\left(\nu+\frac{\frac{\bar{\sigma}_0^2}{\bar{I}_0}+K\bar{I}_0}{\frac{\bar{I}_0^2}{\bar{\sigma}_0^2+\bar{I}_0^2}}\right) \, \Bigl(\frac{\bar{I}_0^2}{\bar{\sigma}_0^2+\bar{I}_0^2} \Bigr)^{-1}.$$
Taking logarithms we obtain $\log \gamma \leq \Theta(\log K) + \nu + K \Theta(1)$, which implies that $\log \gamma \leq \nu + \Theta(K)$ and completes the proof.
\end{proof}

\section{}\label{app:C} 
Our goal in this  Appendix is to prove Lemma\,\ref{lem3}, which connects the threshold  $\tilde{\nu}$ to the false-alarm period, $\gamma$. 
In order to provide an elegant proof of this result, we need to adopt an alternative representation of the fusion center policy (that we will use only in this Appendix). Indeed, since the implementation of $\dcus$ requires  only the knowledge of the transmitted messages at the fusion center, 
it is possible  to describe the fusion rule without any reference to the communication times $\{\tau_n^k\}$. Thus, let $z_{n}$ be the $n$th message that arrives at the fusion center and $k_{n}$ the corresponding identity of the sensor which transmitted this message. Of course, since time is discrete, there is non-zero probability that the fusion center may receive messages from two or more sensors concurrently. In this case, we enumerate the simultaneous messages in an arbitrary order and we keep the same order for the labels.

We can then describe the flow of information at the fusion center by the filtration $\{ {\cal{C}}_{n}\}_{n \in \mathbb{N}}$, 
where $\cC_{n}= \sigma ( (z_{1},k_{1}) \ldots, (z_{n},k_{n}))$. For any $n \in \mathbb{N}$ we set 
\begin{align}
\begin{split}
\phi_{n} &:= \log \frac{\Pro_{0}(k_{1}, \ldots, k_{n})}{\Pro_{\infty}(k_{1}, \ldots, k_{n})}\\
 v_{n}  &:= \log \frac{\Pro_{0}(z_{1}, \ldots, z_{n}|k_{1}, \ldots, k_{n})}{\Pro_{\infty}(z_{1}, \ldots, z_{n}|k_{1}, \ldots, k_{n})}.
 \end{split}
 \label{les}
\end{align}
and recalling the definition of the log-likelihood ratios $\oli_{j},\uli_{j}$ in (\ref{lambdas}),  we have
\begin{equation}
v_{n} = \sum_{m=1}^{n} \sum_{j=1}^{d^{k_{m}}} \Bigl[ \overline{\Lambda}_{j}^{k_{m}} \, \ind{z_{m}=j} - \underline{\Lambda}_{j}^{k_{m}} \,  \, \ind{z_{m}=-j} \Bigr].
\end{equation}
Then, the \textit{number of messages} which the fusion center has received 
until an alarm is raised by D-CUSUM is given by the following $\{{\cal{C}}_{n}\}$-stopping time:
\begin{equation} \label{alterfcp}
\tilde{\cN} = \inf \{ n \in \mathbb{N}: v_{n}- \min_{m=1, \ldots, n} v_{m} \geq \tilde{\nu}  \}.
\end{equation}
The process $\{v_n\}$ and the  stopping time $\tilde{\cN}$ are closely related to  $\{\tilde{u}_t\}$ and $\dcus$, respectively. 
Their main difference is that  $\{\tilde{u}_t\}$ and $\dcus$ are expressed in terms of ``physical time'', whereas $\{v_n\}$ and  $\tilde{\cN}$ in terms of number of messages transmitted to the fusion center. If we denote by $\tau_n$ 
the time-instant at which the $n$th message arrives at the fusion center, then we can explicitly specify the following connection between these quantities: $\tilde{u}_{\tau_n}=v_n$ and $\dcus=\tau_{\tilde{\cN}}$. In other words $\tilde{\cN}$ \textit{denotes the number of received messages at the fusion center until stopping at time} $\tilde{\cS}$.

After these definitions, we can now prove Lemma \,\ref{lem3}, which connects $\tilde{\nu}$ to $\gamma$  through an inequality that will be  important for the performance analysis of $\dcus$. For that, recall the definition of $\bar{I}_\infty$ in \eqref{klaver}.

\begin{proof}[Proof of Lemma \ref{lem3}]
We first observe that
\begin{equation} \label{jjj}
\gamma =\Exp_{\infty}[-u_{\dcus}]= K \bar{I}_{\infty}  \, \Exp_{\infty} [\dcus] \geq \bar{I}_{\infty}  \, \Exp_{\infty} [\tilde{\cN}].
\end{equation}
The second equality follows from an application of Wald's identity, whereas the inequality from the fact that $\tilde{\cN}\le K\tilde{\cS}$. Indeed, the maximum number of received messages until stopping at $\tilde{\cS}$ is obtained when at every time instant we have all sensors transmitting a message to the fusion center and this yields $K\tilde{\cS}$. 

From (\ref{jjj}) it is clear that it suffices to prove $\Exp_\infty[\tilde{\cN}] \geq e^{\tilde{\nu}}$. In order to do so, let us define the sequence $\{n_j\}$ of epochs where the CUSUM process $v_{n}- \min_{0\le m\le n} v_{m}$ either returns to zero (restarts) or exceeds $\tilde{\nu}$. This is the classical way to write the CUSUM stopping time as a sum of a random number of components. Specifically, let us define
\begin{align} \label{repeat0}
\begin{split}
n_{j}  &:= \inf \{ n >  n_{j-1} : v_{n} - v_{n_{j-1}} \notin (0, \tilde{\nu})\}\\
\cR &:= \inf\{j \in \mathbb{N}: v_{n_{j}}- v_{n_{j-1}} \geq \tilde{\nu} \}.
\end{split}
\end{align}
Then we clearly have $\tilde{\cN}=n_{\cR}$. Since from one epoch to the next we count at least one additional message, we trivially conclude that $\cR\le\tilde{\cN}$ and, therefore, $\Exp_\infty[\cR]\le\Exp_\infty[\tilde{\cN}]$. We can now claim that it suffices to show that 
\begin{equation} \label{todo}
\Pro_\infty(\cR>j)\ge (1-e^{-\tilde{\nu}})^j, \quad \forall j \in \mathbb{N}.
\end{equation} 
In order to justify this claim, observe first that $\Exp_\infty[\tilde{\cN}]<\infty$, since $\tilde{\cN}$ is a CUSUM stopping time. As a result, 
$\Exp_\infty[\cR]$ is finite as well and consequently (\ref{todo}) implies that 
\begin{equation*}
\Exp_\infty[\tilde{\cN}] \geq \Exp_\infty[\cR] =\sum_{j=0}^\infty\Pro_\infty(\cR>j) \geq \sum_{j=0}^\infty  (1-e^{-\tilde{\nu}})^j \geq e^{\tilde{\nu}}.
\end{equation*}
In order to prove (\ref{todo}), we start with the following observation:
\begin{align}
\begin{split}
\Pro_\infty(\cR>j)&=\Pro_\infty(\cR>j-1;v_{n_j}-v_{n_{j-1}}\le0)\\
&=\Pro_\infty(\cR>j-1)-\Pro_\infty(\cR>j-1;v_{n_j}-v_{n_{j-1}}\ge\tilde{\nu}).
\end{split}
\label{eq:mbifla2}
\end{align}
Let us now set $A:= \{\cR> j-1  \, , \, v_{n_j}-v_{n_{j-1}}\ge\tilde{\nu}\}$. Then, it is clear  that $A \in \cC_{ n_{j}}$ and
with a change of measure $\Pro_{\infty} \mapsto \Pro_{0}$ we obtain
\begin{align} \label{teno}
\Pro_{\infty}(A) &= \int_{A}  \ccL_{n_{j}}^{-1}  \; d\Pro_{0}, \quad \text{where} \quad  \ccL_{n}:= e^{\phi_{n}+v_{n}}, \quad \forall \; n \in \mathbb{N}.
\end{align}
We now argue as follows
\begin{align} \label{ten}
\begin{split}
\Pro_{\infty}(A) &= \int_{A}  \ccL_{n_{j-1}}^{-1}  \,  e^{-( \phi_{n_{j}}- \phi_{n_{j-1}}) - (v_{n_{j}}-v_{n_{j-1}})}     \; d\Pro_{0} \\
 &\leq e^{-\tilde{\nu}} \, \int_{A}  \ccL_{n_{j-1}}^{-1}  \,  e^{-( \phi_{n_{j}}- \phi_{n_{j-1}})}     \; d\Pro_{0} \\
 &\leq e^{-\tilde{\nu}} \, \int_{\cR > j-1}  \ccL_{n_{j-1}}^{-1}  \,  e^{-( \phi_{n_{j}}- \phi_{n_{j-1}})}     \; d\Pro_{0} \\
 &=  e^{-\tilde{\nu}} \, \int_{\cR > j-1}  \ccL_{n_{j-1}}^{-1}   \, \Exp_{0}[  e^{-( \phi_{n_{j}}- \phi_{n_{j}-1})}  | \mathcal{C}_{n_{j-1}} ] \;   d\Pro_{0}.
 \end{split}
 \end{align}
The first inequality is due to the fact that $v_{n_j}-v_{n_{j-1}}\ge\tilde{\nu}$ on the event $A$. 
The second inequality holds because $A \subset \{\cR > j-1\}$, whereas the last equality follows from 
the law of iterated expectation and the fact that $\{\cR > j-1\} \in \cC_{n_{j-1}}$ and $\ccL_{n_{j-1}}^{-1}$ is a $\cC_{n_{j-1}}$-measurable random variable. 

As a likelihood ratio process,   $\{e^{-\phi_{n}}\}_{n \in \mathbb{N}}$ is a positive $(\Pro_{0}, \cC_{n})$-martingale and, consequently supermartingale. As a result, we can apply the  Optional Sampling Theorem and obtain  
\begin{equation} \label{os}
\Exp_{0}[  e^{-( \phi_{n_{j}}- \phi_{n_{j}-1})}  \, | \,  \mathcal{C}_{n_{j-1}} ] \leq 1.
\end{equation}
 Then, it is clear with a change of measure $\Pro_{\infty} \mapsto \Pro_{0}$ that (\ref{ten}) reduces to 
\begin{align} \label{eleven}
\begin{split}
\Pro_{\infty}(A) &\leq  e^{-\tilde{\nu}} \, \int_{\cR > j-1}  \ccL_{n_{j-1}}^{-1} \;   d\Pro_{0} = e^{-\tilde{\nu}} \, \Pro_{\infty}( \cR > j-1) .
 \end{split}
 \end{align}
Substituting the outcome of \eqref{eleven} in \eqref{eq:mbifla2} and applying it repeatedly yields 
\begin{equation*}
\Pro_\infty(\cR>j)\ge(1-e^{-\tilde{\nu}}) \, \Pro_\infty(\cR>j-1)\ge(1-e^{-\tilde{\nu}})^j,
\end{equation*}
which completes the proof. 

\end{proof}

\section{}\label{app:D} 

\begin{proof}[Proof of Lemma\,\ref{lem:0}]
From the definition of $\oli_{j}$ in \eqref{lambdas} and a  change of measure $\Pro_{\infty} \mapsto \Pro_{0}$ we have
\begin{align*}
e^{-\oli_{j}} &= \frac{\Pro_{\infty}( z_{1}^k=j)}{\Pro_{0}( z_{1}^k=j)} = e^{-\oDi_{j}} \, \frac{\Exp_{0}[ e^{-(\ell^k_{1}- \oDi_{j})} \ind{z_{1}^k=j}]}{\Pro_{0}( z_{1}^k=j)} =  e^{-\oDi_{j}} \, \Exp_{0}[ e^{-(\ell^k_{1}- \oDi_{j})} \, | \,   z_{1}^k=j].
\end{align*}
Taking logarithms we obtain the first equality in (\ref{import}), whereas the second one can be shown in a similar way. It is clear that
$\oRi_{j}, \uRi_{j}>0$ for every $1 \leq j  \leq d$ and that $\oRi_{j}, \uRi_{j} \leq \epsilon^{k}$ for every $1 \leq j  \leq d-1$,
thus, it remains to prove (\ref{import2}). We will prove only the first relationship in it, as the second one can be shown in a similar way.

From the conditional Jensen inequality  we obtain
\begin{align} \label{qqq}
\oRi_{d} &\leq \Exp_{0}[ \ell^k_{1}- \oDi_{d} \, | \, z_{1}^k=d ]  = \frac{\Exp_{0}[ (\ell^k_{1}- \oDi_{d}) \, \ind{z_{1}^k=d}]}{\Pro_{0}(z_{1}^k=d)}
\end{align}
and from (\ref{eq:levels}) we have 
\begin{equation} \label{ov111}
\Pro_{0}(z_{1}^k=d)= \Pro_{0}(z_{1}^k=d|z_{1}^{k}>0) \,  \Pro_{0}(z_{1}^{k}>0) = \frac{\Pro_{0}(z_{1}^{k}>0)}{d} = \frac{1-o(1)}{d},
\end{equation}
where $o(1)$ is a term that vanishes as $\oDi, \uDi \rightarrow \infty$ and does not depend on $d$.  

Moreover, since  $\oDi_{d}= \oDi +\oei_{d-1}$ we have 
\begin{align*} 
\Exp_{0}[ (\ell^k_{1}- \oDi_{d}) \, \ind{z_{1}^k=d}] &= \int_{\oei_{d-1}}^{\oei_{d}} \Pro_0(\ell_{1}^{k}>\oDi+x) \, dx  \\
&\leq \int_{\oei_{d-1}}^{\oei_{d}} \Pro_0(\ell_{1}^{k}>\oDi+x | \ell_{1}^{k} \geq \oDi ) \, dx,
\end{align*}
Setting  $D:=\Exp_0[((u_1^k)^+)^2]/I_{0}^{k}$,  which is clearly a finite quantity since $\Exp_0[(u_1^k)^2]<\infty$, 
(recall also that $I_{0}^{k}= \Exp_0[u_1^k]$), we can apply \cite[Theorem\,4, Eq. (13)]{lorden} and obtain the following upper bound for the probability inside the integral:
\begin{align*} 
\Pro_0(\ell_{1}^{k}-\oDi>x | \ell_{1}^{k} \geq \oDi ) &\leq\frac{1}{I_0^k}\left(\frac{\oDi+D}{\oDi+x}\right)\Exp_0[(2u_1^k-x)\ind{u_1^k\geq x}] \\
&\leq \Theta(1) \, \Exp_0[u_1^k\ind{u_1^k\geq x}] ,
\end{align*}
where $\Theta(1)$ is a term that is independent of $d$ and is bounded from above and below as $\oDi, \uDi \rightarrow \infty$. 
Then, applying Fubini's theorem we obtain 
\begin{align} \label{ov222}
\begin{split}
\Exp_{0}[ (\ell^k_{1}- \oDi_{d}) \, \ind{z_{1}^k=d}]  &\leq \Theta(1) \,   \int_{\oei_{d}-1}^{\oei_{d}} \Exp_0[u_1^k\ind{u_1^k\geq x}]dx \\
&= \Theta(1)  \, \Exp_0[u_1^k(u_1^k-\oei_{d-1})^+]  \leq \Theta(1) \,  \Exp_0[(u_1^k)^2\ind{u_1^k>\oei_{d-1}}].
\end{split}
\end{align}
Combining (\ref{qqq}), (\ref{ov111}) and (\ref{ov222}) completes the proof.
\end{proof}

\begin{proof} [Proof of Lemma\,\ref{lem:1}]
From  (\ref{ells}) and (\ref{import}) we have  
\begin{align} \label{error}
\begin{split}
\ell^k_1- \tilde{\ell}^k_1 
                  &= \sum_{j=1}^{d}\Bigl[ (\ell^k_1-\oDi_{j} -\oRi_j) \ind{z_{1}^{k}= j} + (\ell^k_1+ \uDi_{j} +\uRi_j) \ind{z_{1}^{k}=-j}\Bigr] \\
                  &\leq \sum_{j=1}^{d}\Bigl[ (\ell^k_1-\oDi_{j}) \ind{z_{1}^{k}= j} + \uRi_j \ind{z_{1}^{k}=-j}\Bigr] \\
                  &\leq  \sum_{j=1}^{d-1} \Bigl[  \epsilon^{k} \ind{z_{1}^{k}= j} + \epsilon^{k} \ind{z_{1}^{k}=-j}\Bigr] 
                  + (\ell^k_1-\oDi_{d}) \ind{z_{1}^{k}= d^} + \uRi_{d} \ind{z_{1}^{k}=-d}.                                                                  \end{split}                  
\end{align}
where the  first inequality holds because $\oRi_{j}>0$  and $\ell^k_n+ \uDi_{j}<0$ on $\{z_{n}^{k}=-j\}$ and the second one because $\ell_{n}^{k} -\oDi_{j}  \leq \epsilon^{k}$ on $\{z_{n}^{k}=j\}$ and $\uRi_{j} \leq \epsilon^{k}$  for every $1 \leq j  \leq d-1$.

From (\ref{eq:levels}) it follows  that for any $1 \leq j \leq d$ 
\begin{align*} 
\Pro_{0}(z_{1}^{k}= j) &\leq \Pro_{0}(z_{1}^{k}= j| z_{1}^{k}>0)= 1/d ,  \\
\Pro_{0}(z_{1}^{k}=-j) &= \Exp_{\infty}[e^{\ell_{1}^{k}} \ind{z_{1}^{k}=-j}] \leq \Pro_{\infty}(z_{1}^{k}=-j) \leq \Pro_{\infty}(z_{1}^{k}=-j|  z_{1}^{k}<0 )= 1/d,
\end{align*}
therefore, taking expectations in (\ref{error}) we obtain
$$ 
\Exp_{0}[\ell^k_1- \tilde{\ell}^k_1]  \leq 2\epsilon^{k} \frac{d-1}{d} +  \Exp_{0}[ (\ell^k_1-\oDi_{d}) \ind{z_{1}^{k}= d}] + \frac{\uRi_{d}}{d}.
$$ 
Using now (\ref{ov222})and (\ref{import2}), we obtain upper bounds for the second and third term of right-hand side, respectively, which lead to (\ref{thetak}). This expression implies that $\theta^{k} \rightarrow 0$ as $d \rightarrow \infty$, since $\epsilon^{k} \rightarrow 0$ as $d \rightarrow \infty$ and $u_{1}^{k}$ has a finite second moment. 
\end{proof}

\section{}\label{app:E} 
In this Appendix, we state and prove Lemma  \ref{lem:8},  which is used in the proof of Theorem \,\ref{th:2}. In order to do so, 
we  need a very useful for our purposes, asynchronous version of Wald's identity (Lemma\,\ref{lem:4}), as well as  the following lemma.
We  set:
$$\Lambda_{\max} :=\max_{1 \leq k \leq K} \max_{1 \leq j \leq d} \max \{\oli_{j}, \uli_{j}\}.$$

\begin{lemma}\label{lem:7}
If $\Exp_{i}[(u_{1}^{k})^{2}]<\infty$ for every $1 \leq k \leq K$, then  as $\D \rightarrow \infty$ we have
$\Lambda_{\max} =\Theta(\D)$ and $$\min_{1 \leq k \leq K} \Exp_{0}[\tilde{\ell}_{1}^{k}] \geq \Theta(\D).$$
\end{lemma}

\begin{proof}
From Lemma \ref{lem:0} it is clear that $\oRi_{j}, \uRi=O(1)$ and, consequently, $\oli_{j}, \uli_{j}=\Theta(\D)$ as $\D \rightarrow \infty$ for every $j=1, \ldots,d$, which proves that  $\Lambda_{\max} =\Theta(\D)$. Furthermore,  since $\oli_{j} \geq \oDi_{j} \geq \oDi$ and $\uli_{j} \leq \Lambda_{\max}$ we have
\begin{align*}
\Exp_{0}[\tilde{\ell}_{1}^{k}] &=  \sum_{j=1}^{d} [ \oli_{j} \, \Pro_{0}(z_{1}^{k}=j)- \uli_{j} \, \Pro_{0}(z_{1}^{k}=-j) ] \\
&\geq  \oDi \, \Pro_{0}(z_{1}^{k}>0) - \Lambda_{\max} \, \Pro_{0}(z_{1}^{k}<0) \\
&= \oDi  - (\oDi+  \Lambda_{\max} ) \, \Pro_{0}(z_{1}^{k}<0),
\end{align*}
thus, it suffices to show that $\Pro_{0}(z_{1}^{k}<0)=o(1/\D)$. Indeed,  with a change of measure we have
$$\oDi \Pro_{0}(z_{1}^{k}<0) = \oDi \, \Exp_{\infty}[e^{\ell_{1}^{k}} \,  \ind{\ell_{1}^{k}<-\uDi}] \leq \oDi \,  e^{-\uDi}$$
and the upper bound clearly goes to 0 as $\D \rightarrow \infty$. 
\end{proof}

\begin{lemma}\label{lem:4}
Consider a generic sequence $\{\zeta_{n}^k\}$, where each $\zeta_{n}^k$ is an arbitrary (Borel) function of the triplet 
$(\tau_{n}^k-\tau_{n-1}^k, z_{n}^k, \ell_{n}^{k})$. Thus,  $\{\zeta_{n}^k\}$ is a sequence of independent and identically distributed random variables under both $\Pro_{0}$ (and $\Pro_{\infty}$).
If  $\cT$ is a $\Pro_{0}$-integrable $\{\cF_t\}$-stopping time and  $\Exp_{0}[|\zeta_{1}^k|]< \infty$, then
\begin{equation} \label{abra1}
\Exp_{0} \left[ \sum_{n=1}^{m_{\cT}^k+1} \zeta_{n}^k \right] = \Exp_{0} [ m_{\cT}^k+1] \Exp_{0}[\zeta_{1}^k].
\end{equation}
If moreover $\zeta_{n}^k\geq 0$, then
\begin{equation} \label{abra2}
\Exp_{0} \left[ \sum_{n=1}^{m_{\cT}^k} \zeta_{n}^k \right] \leq (\Exp_{0} [m_{\cT}^k]+1) \; \Exp_{0}[\zeta_{1}^k].
\end{equation}
Finally, if  $|\zeta_{n}^k|\leq M^{k}$, where $M^{k}$ is some finite constant, then
\begin{equation} \label{abra3}
\Exp_{0} \left[ \sum_{n=1}^{m_{\cT}^k} \zeta_{n}^k \right] \geq \Exp_{0} [m_{\cT}^k] \; \Exp_{0}[\zeta_{1}^k] - 2M^{k}.
\end{equation}
\end{lemma}
\begin{proof}
The proof can be found in \cite{fel}.
\end{proof}


\begin{lemma} \label{lem:8}
If $\Exp_{i}[(u_{1}^{k})^{2}]<\infty$ for every $1 \leq k \leq K$, then as $\D \rightarrow \infty$
\begin{align}
\tilde{u}_{\tilde{\cS}} &\leq \log\gamma+ K \Theta(\D) \label{flas1} \\
\Exp_0[u_{\tilde{\cS}}-\tilde{u}_{\tilde{\cS}}] &\leq  K \Theta(\D) +  \frac{\theta}{\Theta(\D)}  \, \log \gamma. \label{flas2}
\end{align}
\end{lemma}

\begin{proof}
In order to prove (\ref{flas1}), it  suffices to observe that the overshoot $\tilde{y}_{\tilde{\cS}}-\tilde{\nu}$ cannot be larger than $K \Lambda_{\max}$,
therefore, 
$$
\tilde{u}_{\tilde{\cS}} \leq \tilde{y}_{\tilde{\cS}} \leq  \tilde{\nu}+ K \, \Lambda_{\max}  \leq \log \gamma +  K \Theta(\D) ,
$$
where the last inequality follows from Lemmas \ref{lem3} and \ref{lem:7}. In order to prove (\ref{flas2}), we observe that for any  $t$ and $k$ we have
\begin{align*}
u^k_t-\tilde{u}^k_t &=u^k_t-u^k_{\tau^k(t)}+u^k_{\tau^k(t)}-\tilde{u}^k_{\tau^k(t)} \leq  \oDi  + \sum_{n=1}^{m^k_t} [ \ell_{n}^{k}- \tilde{\ell_{n}^{k}}],
\end{align*}
If we now replace  $t$ with $\tilde{\cS}$, take expectations with respect to $\Pro_0$ and  apply (\ref{abra2}) and  Lemma\,\ref{lem:1} we obtain 
$$
\Exp_0[u^k_{\tilde{\cS}}-\tilde{u}^k_{\tilde{\cS}}] \leq \oDi+ (\Exp_0[m^k_{\tilde{\cS}}]+1) \, \theta^{k}
=  \oDi+ \theta^{k}+ \theta^{k} \Exp_0[m^k_{\tilde{\cS}}].
$$
Since from (\ref{thetak}) it is clear that $\theta^{k}=\calo(1)$ as $\D \rightarrow \infty$,  summing over $k$ we obtain
\begin{align}  \label{sss}
\begin{split}
\Exp_0[u_{\tilde{\cS}}-\tilde{u}_{\tilde{\cS}}] 
&\leq \sum_{k=1}^{K} (\oDi+\theta^{k}) + \sum_{k=1}^{K} \theta^{k} \, \Exp_0[m^{k}_{\tilde{\cS}}] 
\leq  K \Theta(\D) +  \theta \, \Exp_0[m_{\tilde{\cS}}] ,
\end{split}
\end{align}
where $m_{t}:=\sum_{k=1}^{K} m_{t}^{k}$. Now, it is obvious that $|\tilde{\ell}_{n}^{k}| \leq \Lambda_{\max}$ for every $n$ and $k$, therefore applying (\ref{abra3}) we have
$$
\Exp_{0}[\tilde{u}^{k}_{\tilde{\cS}}]= \Exp_{0} \Bigl[ \sum_{n=1}^{m_{\tilde{\cS}}^k} \tilde{\ell}_{n}^{k} \Bigr] \geq \Exp_{0}[m^{k}_{\tilde{\cS}}] \, \Exp_{0}[\tilde{\ell}_{1}^{k}] -2 \Lambda_{\max}.
$$
Thus, summing over $k$ we obtain
\begin{align*}
\Exp_{0}[\tilde{u}_{\tilde{\cS}}] 
& \geq   \Bigl(\min_{1 \leq k \leq K} \Exp_{0}[\tilde{\ell}_{1}^{k}] \Bigr) \, \Exp_{0}[m_{\tilde{\cS}}] - 2 \,  K \, \Lambda_{\max} 
\end{align*}
and, consequently, 
\begin{align*}
\Exp_{0}[m_{\tilde{\cS}}] &\leq \frac{\Exp_{0}[\tilde{u}_{\tilde{\cS}}] + 2 K \, \Lambda_{\max}}{\min_{1 \leq k \leq K} \Exp_{0}[\tilde{\ell}_{1}^{k}]} 
\leq \frac{ \log \gamma + K \Theta(\D)}{\Theta(\D)}=  \frac{ \log \gamma}{\Theta(\D)}+  K \Theta(\D).
\end{align*}
where the second inequality is  due to (\ref{flas1}) and Lemma \ref{lem:7}. Combining the latter relationship with (\ref{sss}) we obtain the desired result. 
\end{proof}

\end{document}